\documentclass[reqno]{amsart}

\usepackage{tikz}
\usetikzlibrary{decorations.pathreplacing,shapes,arrows}

\newcommand{\edgedraw}[2]{\draw (#1)--(#2);}

\newcommand{\dashedgedraw}[2]{\draw [dashed] (#1)--(#2);}
\newcommand{\thickedgedraw}[2]{\draw [ultra thick] (#1)--(#2);}
\usetikzlibrary{decorations.markings}
\usetikzlibrary{arrows,matrix}
\usepgflibrary{arrows}

\usepackage{mathrsfs}
\usepackage{latexsym}
\usepackage{graphicx}
\usepackage{epic,eepic,ecltree} 

\input xy
\xyoption{all}

\usepackage{xy} 

\usepackage{verbatim}
\usepackage{enumerate}

\def\Y{{\mathcal Y}}
\def\X{{\mathcal X}}

\newcommand{\IG}{IG}

\newcommand{\dedge}[1]{\ar@{--}[#1]}
\newcommand{\edge}[1]{\ar@{-}[#1]}
\newcommand{\lulab}[1]{\ar@{}[l]_<<{#1}}
\newcommand{\rulab}[1]{\ar@{}[r]^<<{#1}}
\newcommand{\ldlab}[1]{\ar@{}[l]^<<{#1}}
\newcommand{\rdlab}[1]{\ar@{}[r]_<<{#1}}
\newcommand{\node}{*+[o][F-]{ }}

\renewcommand{\theenumi}{\roman{enumi}}
 \newcommand{\gh}{\mathscr{H}}
  \newcommand{\gr}{\mathscr{R}}
  \newcommand{\gl}{\mathscr{L}}
  
  \newcommand{\gd}{\mathscr{D}}
  \newcommand{\gj}{\mathscr{J}}

    \newcommand{\F}{Q}

 \newcommand{\rank}[1]{\mathrm{rank}{(#1)}}

 \DeclareMathOperator\Coll{Col} \DeclareMathOperator\Roww{Row}

\newtheorem{thm}{Theorem}

\newtheorem{cor}{Corollary}
\newtheorem{defn}{Definition}
\newtheorem{prop}{Proposition}
\newtheorem{lem}{Lemma}

\theoremstyle{definition}

\newtheorem*{acknowledgement}{Acknowledgement}

\numberwithin{equation}{section}

\begin{document}

\title[Free idempotent generated semigroups: full linear monoid]%
{Maximal subgroups of free idempotent generated semigroups over the full linear monoid}

\author{IGOR DOLINKA}
\address{Department of Mathematics and Informatics, University of Novi Sad, Trg Dositeja Obradovi\'ca 4, 21101 Novi Sad,
Serbia} 
\email{dockie@dmi.uns.ac.rs}
\thanks{The research of the first author is supported by the Ministry of Education and Science of the Republic of Serbia
through Grant No.174019, and by a grant (Contract 114--451--2002/2011) of the Secretariat of Science and Technological
Development of the Autonomous Province of Vojvodina.}

\author{ROBERT D. GRAY}
\address{Centro de \'{A}lgebra da Universidade de Lisboa, Av. Prof. Gama Pinto, 2,  1649-003 Lisboa,  Portugal}
\email{rdgray@fc.ul.pt}
\thanks{This work was developed within the project POCTI-ISFL-1-143 of CAUL, supported by FCT}


\begin{abstract}
We show that the rank $r$ component of the free idempotent generated semigroup of the biordered set of the full linear
semigroup full of $n \times n$ matrices over a division ring $Q$ has maximal subgroup isomorphic to the general linear
group $GL_r(Q)$, where $n$ and $r$ are positive integers with $r < n/3$.
\end{abstract}

\maketitle

\section{Introduction}
\label{sec_introduction}

The full linear monoid of all $n \times n$ matrices over a field (or more generally a division ring) is one of the most
natural and well studied of semigroups. This monoid plays an analogous role in semigroup theory as the general linear
group does in group theory, and the study of linear semigroups
 \cite{okninski1998}  is important in a range of areas  such as the representation theory of semigroups
 \cite{AMSV}, \cite[Chapter 5]{CP}, Putcha--Renner theory of linear algebraic monoids (monoids closed in the Zariski topology)
 \cite{Putcha, Renner2005, Solomon}, and the theory of finite monoids of Lie type \cite{putchasemisimple,LieType,Putcha1,Putcha2}.

The full linear monoid $M_n(Q)$ (where $Q$ is an arbitrary division ring) is an example of a so-called (von Neumann)
regular semigroup. In 1979 Nambooripad published his foundational paper \cite{nambooripad79} on the structure of
regular semigroups, in which he makes the fundamental observation that the set of idempotents $E(S)$ of an arbitrary
semigroup carries a certain abstract structure of a so-called biordered set (or regular biordered set in the case of
regular semigroups). He provided an axiomatic characterisation of (regular) biordered sets in his paper, and later
Easdown extended this to arbitrary (non-regular) semigroups \cite{easdown85} showing that each abstract biordered set
is in fact the biordered set of a suitable semigroup. Putcha's theory of monoids of Lie type shows that one can view
the biordered set of idempotents of a reductive algebraic monoid as a generalised building \cite{Putcha}, in the
sense of Tits. Thus, in the context of reductive algebraic monoids, a natural geometric structure is carried by the
biordered set of idempotents. We shall not need the formal definition of biordered set here, for more details of the
theory of abstract biordered sets we refer the reader to \cite{higgins92}.

The study of biordered sets of idempotents of semigroups is closely related with the study of idempotent generated
semigroups. Here a semigroup is said to be idempotent generated if every element is expressible as a product of
idempotents of the semigroup. Such semigroups are in abundance in semigroup theory. For instance, every non-invertible
matrix of $M_n(Q)$ is expressible as a product of idempotent matrices \cite{erdos67, laffey83}, and the same result is
true for the full transformation semigroup of all maps from a finite set to itself \cite{howie66}. More recently, in a
significant extension of Erdos's result, Putcha \cite{putcha06} gave necessary and sufficient conditions for a
reductive linear algebraic monoid to have the property that every non-unit is a product of idempotents.
Idempotent generated semigroups have received considerable attention in the literature, in part because of the large
number of semigroups that occur in nature that have this property, and also because of the universal property that they
possess: every semigroup embeds into an idempotent generated semigroup, and if the semigroup is (finite) countable it
can be embedded in a (finite) semigroup generated by $3$ idempotents.

For a fixed abstractly defined biordered set $E$, the collection of all semigroups whose biordered set of idempotents
is (biorder) isomorphic to $E$ forms a category when one restricts morphisms to those that are one-to-one when
restricted to the set of idempotents. There is an initial object in this category, called the \emph{free idempotent
generated semigroup over $E$} and denoted $IG(E)$, that thus maps onto every idempotent generated semigroup with
biordered set $E$ via a morphism that is one-to-one on idempotents.
Clearly an important step towards understanding the class of semigroups with fixed biordered set of idempotents $E$ is
to study the free objects $IG(E)$. For semigroup-theoretic reasons, much of the structure of $IG(E)$ comes down to
understanding the structure of its maximal subgroups. Until recently, very little was known about maximal subgroups of
free idempotent generated semigroups. In fact, in all known cases, all such maximal subgroups had turned
out to be free groups, and in \cite{mcelwee02} it was conjectured that this would always be the case. However, in 2009
Brittenham, Margolis and Meakin \cite{Brittenham2009} gave a counterexample to this conjecture by showing that the free
abelian group of rank $2$ arises as a maximal subgroup of a free idempotent generated semigroup. The proof in \cite{Brittenham2009} makes use of new topological tools introduced for the study of
maximal subgroups of $IG(E)$. In this new theory in a natural way a $2$-complex, called the Graham--Houghton $2$-complex
$GH(E)$, is associated to a regular biordered set $E$ (based on work of Nambooripad \cite{nambooripad79}, Graham
\cite{Gr} and Houghton \cite{Hough}) and the maximal subgroups of $IG(E)$ are the fundamental groups of the connected
components of $GH(E)$. The $2$-cells of $GH(E)$ correspond to the singular squares of $E$ defined by Nambooripad in
\cite{nambooripad79}.

More recently, in \cite{grayta}, an alternative approach to the study of maximal subgroups of free idempotent generated
semigroups was introduced. Using Reidemeister--Schreier rewriting methods originally developed in \cite{Ruskuc1999},
together with methods from combinatorial semigroup theory (that is, the study of semigroups by generators and
relations), a presentation for an arbitrary maximal subgroup of $IG(E)$ was given in \cite[Theorem~5]{grayta}. Then
applying this result it was shown that, in fact, every abstract group arises as a maximal subgroup of $IG(E)$, for an
appropriately chosen biordered set. Moreover, it was shown that every finitely presented group is a maximal subgroup of
a free idempotent generated semigroup over a finite biordered set $E$.

Other recent work in the area includes \cite{dolinka} where free idempotent generated semigroups over bands are investigated, and it is shown that there is a regular band $B$ such that $IG(B)$ has a maximal subgroup isomorphic to the free abelian group of rank $2$.

However, the structure of the maximal subgroups of free idempotent generated semigroups on naturally occurring
biordered sets, such as the biordered set of the full linear monoid $M_n(Q)$ over a division ring $Q$, remained far
from clear. In a recent paper \cite{margolismeakinip} Brittenham, Margolis and Meakin further developed their
topological tools to study this problem. The main result of \cite{margolismeakinip} shows that the rank 1 component of
the free idempotent generated semigroup of the biordered set of a full matrix monoid of size $n \times n, n>2$, over a
division ring $Q$ has maximal subgroup isomorphic to the multiplicative subgroup of $Q$. This result provided the first
natural example of a torsion group that arises as a maximal subgroup of a free idempotent generated semigroup on some
finite biordered set, answering a question raised in \cite{easdownta}. It is remarked in \cite{margolismeakinip} that
the methods used there seem difficult to extend to higher ranks. Here we shall extend their result, showing that
general linear groups arise as maximal subgroups in higher rank components.

As mentioned above, the \emph{free idempotent generated semigroup over $E$} is the universal object in the category of
all idempotent generated semigroups whose biordered sets of idempotents are isomorphic to $E$.
Given a semigroup $S$ with set of idempotents $E=E(S)$ the free idempotent generated semigroup over $E$ is the semigroup defined by the following presentation.
\begin{equation}
\label{eq1} \IG(E)=\langle E\:|\: e\cdot f= ef\ (e,f\in E,\ \{ e,f\}\cap\{ef,fe\}\neq\emptyset)\rangle.
\end{equation}
(It is an easy exercise to show that if, say, $fe\in\{e,f\}$ then $ef\in E$. In the defining relation $e\cdot f=ef$ the
left hand side is a word of length $2$, and $ef$ is the product of $e$ and $f$ in $S$, i.e. a word of length $1$.) The
idempotents of $S$ and $\IG(E)$ are in natural one-one correspondence (see Proposition~\ref{prop_basic}\eqref{IG2}
below), and we will identify the two sets throughout. We may now state our main result.
\begin{thm} \label{thm_main}
Let $n$ and $r$ be positive integers with $r < n/3$, let $E$ be the biordered set of idempotents of the full linear
monoid $M_{n}(\F)$ of all $n \times n$ matrices over an arbitrary division ring $\F$, and let $W$ be an idempotent
matrix of rank $r$.
Then the maximal subgroup of $IG(E)$ with identity $W$ is isomorphic to the general linear group $GL_{r}(\F)$.
\end{thm}
Observe here that the condition $r < n/3$ forces $n \geq 4$. Theorem~\ref{thm_main} extends the main result of
\cite{margolismeakinip} where Theorem~\ref{thm_main} is proved in the case $r=1$ and $n \geq 3$. In particular
Theorem~\ref{thm_main} shows that arbitrary general linear groups arise as maximal subgroups of naturally occurring
biordered sets.
An analogous result for the full transformation semigroup $T_n$ of all mappings from the set $\{1,\ldots,n\}$ to itself
under composition was recently established in \cite{GR2} where it is shown how the standard Coxeter presentation for
the symmetric group $S_r$ is encoded by the set of all idempotents with image size $r$.
Our methods do not extend to higher values of $r$, and the problem of describing the maximal subgroups in those cases
remains open (see Section~\ref{sec_concluding} for further discussion of this).

The proof of Theorem~\ref{thm_main} is broken down into stages. For each stage of the proof the initial algebraic
problem will be recast in purely combinatorial terms. At the heart of the proof will be the detailed analysis of
various connectedness conditions satisfied by the structure matrices of the principal factors of the monoid $M_n(Q)$.
Several different notions of connectedness arise, the first of which will be analysed using a coloured bipartite graph
representation, closely related to the Graham--Houghton graphs employed in \cite{Brittenham2009,margolismeakinip}, while later connectedness conditions concern graphs
obtained in a natural way from occurrences of symbols arising in multiplication tables of semigroups of matrices. The
fact that these notions of connectedness are central to the proof reflects the natural geometric and topological
structure underlying the problem, as explored in detail in \cite{Brittenham2009,margolismeakinip}.

The paper is structured as follows. In Section~\ref{sec_preliminaries} we give the necessary background on matrix
semigroups over division rings and on free idempotent generated semigroups, and then we go on to apply
results from \cite{Brittenham2009,margolismeakinip} to
write down a presentation for an arbitrary maximal subgroup of a rank $r$ idempotent in
$IG(E(M_n(Q)))$. The remainder of the paper is concerned with proving that, when $r < n/3$, the group that this
presentation defines is $GL_r(Q)$. This proof is broken down into three main steps which are explained in
Section~\ref{sec_outline}. We work through the main steps of the proof over Sections~\ref{sec_Ilabels},
\ref{sec_multtables}, \ref{sec_lambdalabels} and \ref{sec_uncovering}. Finally, in Section~\ref{sec_concluding} we
discuss some open problems and possible directions for future research.

\section{Preliminaries}
\label{sec_preliminaries}

\subsection*{Matrix semigroups}

Throughout this paper, $Q$ will be an arbitrary fixed division ring, $M_n(Q)$ will denote the full linear monoid of $n
\times n$ matrices over $Q$, and we shall use $M_{m \times l}(Q)$ to denote the set of all $m \times l$ matrices over
$Q$, for positive integers $m$, $l$. We let $GL_n(Q)$ denote the general linear group of all invertible $n \times n$
matrices over $Q$, which is, of course,  the group of units of the monoid $M_n(Q)$. We also choose and fix an arbitrary
idempotent $W$ of $M_n(Q)$ of rank $r < n/3$. Since we are working over a division ring, which might not be
commutative, some care needs to be taken here with notions like the rank of a matrix. Linear combinations of rows will
always be taken using left scalar multiplication, and linear combinations of columns will be taken using right scalar
multiplication. Thus by the row space $\Roww A$ of a matrix $A$ we shall mean left row space, by the column space $\Coll A$ of
$A$  we shall mean right column space, and by the rank of a matrix we mean the left row rank of the matrix, which is
equal to its right column rank.

With $E$ equal to the biordered set of idempotents of $M_n(Q)$ our aim is to prove that the maximal subgroup of the
free idempotent generated semigroup $IG(E)$ with identity $W$ is isomorphic to the general linear group $GL_r(Q)$.
Since any pair of maximal subgroups in the same $\gd$-class of a semigroup are isomorphic, by
Proposition~\ref{prop_basic}\eqref{IG3} below it follows that without loss of generality we may take
\[
W=
\left[%
\begin{array}{cc}
  I_{r} & 0 \\
  0 & 0 \\
\end{array}%
\right]
\]
where $I_r$ denotes the $r \times r$ identity matrix.

In general, important structural information about a semigroup may be obtained by studying its ideal structure. Since
their introduction in \cite{Green1951}, Green's relations have provided a powerful tool for the investigation of the
ideal structure of semigroups. Recall that two elements $s$ and $t$ of a semigroup $S$ are said to be $\gr$-related if
they generate the same principal right ideal, $\gl$-related if they generate the same principal left ideal, and
$\gj$-related if they generate the same principal two-sided ideal, that is
\begin{align*}
s \gr t \Leftrightarrow sS \cup \{ s \} = tS \cup \{ t \}, & &
s \gl t \Leftrightarrow Ss \cup \{ s \} = St \cup \{ t \},
\end{align*}
\[
s \gj t \Leftrightarrow SsS \cup sS \cup Ss \cup \{ s \} = StS \cup tS \cup St \cup \{ t \}.
\]
In addition, we have the relations $\gh = \gr \cap \gl$ and $\gd = \gr \circ \gl = \gl \circ \gr$ which is the join of $\gr$ and $\gl$ in the lattice of equivalence relations on $S$. Given an element $a \in S$ we use $R(a,S)$ to denote its $\gr$-class, and similarly we use the notation $L(a,S)$, $J(a,S)$, $H(a,S)$ and $D(a,S)$.
Let $e \in S$ be an idempotent. The set $e S e$ is a submonoid in $S$ with identity element $e$, and it is the largest submonoid of $S$ which has $e$ as identity. The group of units of $e S e$ is the largest subgroup of $S$ with identity $e$, and is called the maximal subgroup of $S$ containing $e$. This maximal subgroup is precisely the $\gh$-class $H(e,S)$ of $S$ that contains the idempotent $e$. More background on Green's relations and their importance in semigroup theory may be found, for example, in \cite{howie95}.

One particularly important class are those semigroups that do not have any proper two-sided ideals. A semigroup $S$ is
called simple if its only ideal is $S$ itself, and a semigroup with zero $0 \in S$ is called $0$-simple if $\{ 0 \}$
and $S$ are its only ideals (and $S^2 \neq \{ 0 \}$). A semigroup is called completely ($0$-)simple if it is
($0$-)simple and has ($0$-)minimal left and right ideals, under the natural orders on left and right ideals by
inclusion.

Let $S$ be a completely $0$-simple semigroup. The Rees theorem \cite[Theorem~3.2.3]{howie95} states that $S$ is
isomorphic to a regular Rees matrix semigroup $\mathcal{M}^0[G; I, \Lambda; P]$ over a group $G$, and conversely that
every such semigroup is completely $0$-simple. Here $G$ is a group, $I$ and $\Lambda$ are index sets, $P = (p_{\lambda
i})$ is a regular $\Lambda \times I$ matrix over $G \cup \{ 0 \}$ (where regular means that every row and column of the
matrix contains at least one non-zero entry) called the structure matrix, and $S = \mathcal{M}^0[G; I, \Lambda; P]$ is
the semigroup with elements $(I \times G \times \Lambda) \cup \{ 0 \}$ and multiplication defined by
$(i,g,\lambda)(j,h,\mu) = (i,gp_{\lambda,j}h,\mu)$ if $p_{\lambda,j} \neq 0$, and $0$ otherwise.

The importance of $0$-simple semigroups comes from the way in which they may be viewed as basic building blocks of
arbitrary semigroups. Indeed, given a $\gj$-class $J$ of a semigroup $S$ we can form a semigroup $J^0$ from $J$, called
the principal factor of $S$ corresponding to $J$, where $J^0 = J \cup \{ 0 \}$ and multiplication $*$ is given by $s *
t = st$ if $s, t, st \in J$, and $s * t = 0$ otherwise.
It is well known (see \cite{howie95}) that $J^0$ is then either a semigroup with zero multiplication, or $J^0$ is a $0$-simple semigroup. Recall that a semigroup $S$ is called (von-Neumann) regular if $a \in aSa$ for all $a \in S$. A semigroup is regular if and only if every $\gr$-class (equivalently every $\gl$-class) contains at least one idempotent.

Now let us turn our attention back to the full linear monoid $M_n(Q)$. Linear semigroups have received a lot of
attention in the literature, and much is known about the structure of the full linear semigroup $M_n(Q)$; see
\cite{okninski1998, Putcha}. Let us now recall some of these basic fundamental facts regarding $M_n(Q)$ that we shall
need in what follows.
The semigroup  $S=M_{n}(Q)$ is a (von-Neumann) regular semigroup. More than this, it is completely semisimple meaning
that each of its principal factors is a completely $0$-simple semigroup, and thus by the Rees theorem, each principal
factor of $M_n(Q)$ is isomorphic to some Rees matrix semigroup over a group.

The set of matrices of a fixed rank $r \leq n$ forms a $\gj$-class in the monoid $M_{n}(Q)$. In fact, $\gj = \gd$ in
$M_{n}(Q)$ and for matrices $X,Y \in M_{n}(\F)$, we have
\begin{eqnarray}
\label{eqn_J} X \gd Y & \Leftrightarrow &
GL_{n}(Q)
 \, X \,
  GL_{n}(Q) = GL_{n}(Q)
  \,  Y \,
    GL_{n}(Q) \\
& \Leftrightarrow & \rank{X} = \rank{Y}.
\end{eqnarray}
The maximal subgroups of the $\gd$-class of all matrices of rank $r$ are isomorphic to $GL_{r}(Q)$.  Green's relations
$\gr$ and $\gl$ in $M_{n}(\F)$ are described by
\begin{align}
\label{eqn_R} X \gr Y
 \Leftrightarrow
X  \, GL_{n}(Q) = Y  \, GL_{n}(Q)
 \Leftrightarrow
\Coll X  = \Coll Y,
\end{align}
and
\begin{align}
\label{eqn_L} X \gl Y
 \Leftrightarrow
GL_{n}(Q)  \, X = GL_{n}(Q)  \, Y
 \Leftrightarrow
\Roww X = \Roww  Y.
\end{align}
Let $D_r$ be the $\gd$-class of $M_n(\F)$ consisting of the rank $r$ matrices, where $1 \leq r < n$, and let $D_{r}^0$ be the corresponding
principal factor, which we know is a completely $0$-simple semigroup.
Following \cite{margolismeakinip} we now write down a natural Rees matrix representation for the principal factor $D_r^0$.
At the heart of our proof will be a detailed analysis of the combinatorial properties of the structure matrix $P_r$ of this Rees matrix semigroup.

Recall that a matrix is said to be in \emph{reduced row echelon form} (RRE for short) if the following conditions are satisfied
\begin{itemize}
\item all nonzero rows (rows with at least one nonzero element) are above any rows of all zeros,
\item the leading coefficient (the first nonzero number from the left, also called the pivot) of a nonzero row is always strictly to the right of the leading coefficient of the row above it, and
\item every leading coefficient is 1 and is the only nonzero entry in its column.
\end{itemize}
Given an $r \times q$ matrix $A$ in RRE form we use $LC(A)$ to denote the subset of $\{1,\ldots,q\}$ indexing the
leading columns of the matrix, that is, the columns containing the leading 1s. If $A$ is an $r \times q$ matrix in RRE
form, and if $A$ has rank $r$, then all of the rows must be non-zero and the leading coefficient in every row is $1$,
and therefore $A$ must have exactly $r$ leading columns which are, in order, the transposes of the $1 \times r$
standard basis vectors $ \{ [1,0,\ldots,0], [0,1,\ldots,0], \ldots, [0,0,\ldots,1] \}. $ Therefore, for every $r \times
q$ rank $r$ matrix $A$ in RRE form, $LC(A)$ is an $r$-element subset of $\{1,\ldots,q\}$. Dually, given the transpose
$B$ of a matrix $B^T$ in RRE form, we use $LR(B)$ to denote the set of numbers indexing the leading rows of $B$.

Let $\Y_r$ denote the set of all $r \times n$ rank $r$ matrices in RRE form, and let $\X_r$ denote the set of transposes of elements of $\Y_r$. The structure of $D_{r}^0$ is described in the following theorem
(see \cite{okninski1998}).
\begin{thm} \label{thm_rees}
\begin{sloppypar}
The principal factor $D_{r}^0$ of $M_n(Q)$ is isomorphic to the Rees matrix semigroup ${\mathcal
M}^{0}(GL_{r}(Q);\mathcal{X}_r, \mathcal{Y}_r;P_r)$ where the structure matrix $P_{r} = (P_r(Y,X))$ is defined for  $Y
\in \mathcal{Y}_{r}$, $X\in \mathcal{X}_r$ by $P_r(Y,X) = YX$ if $YX$ is of rank $r$ and $0$ otherwise.
\end{sloppypar}
\end{thm}
Given $X \in \X_r$ and $Y \in \Y_r$ we shall use $R(X)$ to denote the $\gr$-class indexed by $X$, and $L(Y)$ to denote the
$\gl$-class indexed by $Y$. So, the $\gr$-classes of $D_r$ are indexed by $\X_r$, the
$\gl$-classes by $\Y_r$, and the $\gh$-class $R(X) \cap L(Y)$ contains an idempotent if and only if $P_r(Y,X) \neq 0$ which is
true if and only if $YX$ has rank $r$. In this case we use $e_{X,Y}$ to denote the unique idempotent in the group
$\gh$-class $R(X) \cap L(Y)$.

\subsection*{Free idempotent generated semigroups} Let $S$ be a semigroup, let $E = E(S)$ be the set of idempotents of $S$, and let $IG(E)$ be the free idempotent generated semigroup over $E$ defined by the presentation \eqref{eq1}.
Some fundamental basic properties of the semigroup $IG(E)$ are summarised in the following statement.

\renewcommand{\theenumi}{\textsf{(IG\arabic{enumi})}}
\renewcommand{\labelenumi}{\theenumi}

\begin{prop}\label{prop_basic}
Let $S$ be a semigroup and $E=E(S)$.
The free idempotent generated semigroup $\IG(E)$ has the following properties:
\begin{enumerate}[(i)]
\item
\label{IG1}
There exists a natural homomorphism $\phi$ from $\IG(E)$ onto the subsemigroup $S^\prime$ of $S$ generated by $E$.
\item
\label{IG2}
The restriction of $\phi$ to the set of idempotents of $\IG(E)$ is a bijection onto $E$ (and an isomorphism of biordered sets). Thus we may identify those two sets.
\item
\label{IG3}
$\phi$ maps the $\gr$-class (respectively $\gl$-class) of $e\in E$ onto the corresponding class of $e$ in $S^\prime$; this induces a bijection between the set of all $\gr$-classes (resp. $\gl$-classes) in the
$\gd$-class of $e$ in $\IG(E)$ and the corresponding set in $S^\prime$.
\item
\label{IG4}
The restriction of $\phi$ to the maximal subgroup of $\IG(E)$ containing $e\in E$ (i.e. to the $\gh$-class of $e$ in $\IG(E)$) is a homomorphism onto the maximal subgroup of $S^\prime$ containing $e$.
\end{enumerate}
\end{prop}
\begin{proof}
The assertion \eqref{IG1} is obvious; \eqref{IG2} is proved in \cite{nambooripad79} and \cite{easdown85}; \eqref{IG3} is a corollary of \cite{fitzgerald72}; \eqref{IG4} follows from \eqref{IG2}.
\end{proof}

\subsection*{\boldmath A presentation for the maximal subgroup of $IG(E)$ containing $W$}

For the remainder of the article $E$ will denote the set of idempotents of the full linear monoid $M_n(Q)$. Our
interest is in the maximal subgroup $H(W,IG(E))$ where $IG(E)$ is defined by the presentation \eqref{eq1},
and our first task will be to write down a presentation for this group.
The key concept that we need in order to write down such a presentation is the notion of a singular square, a concept which originally goes back to work of Nambooripad \cite{nambooripad79}.

Let $S$ be a semigroup with set of idempotents $E = E(S)$.
An \emph{$E$-square} is a sequence $(e,f,g,h)$ of elements of $E$ with $e \; \gr \; f \; \gl \; g \; \gr \; h \; \gl \; e$. Unless otherwise stated, we shall assume that all $E$-squares are non-degenerate, i.e. the elements $e,f,g,h$ are all distinct. An idempotent $t = t^2 \in E$ \emph{left to right singularises} the $E$-square $(e,f,g,h)$ if
\[
te = e, \; th=h, \; et=f \; \mbox{and} \; ht=g.
\]
Right to left, top to bottom and bottom to top singularisation is defined similarly and we call the $E$-square \emph{singular} if it has a singularising idempotent of one of these types.

A biordered set $E$ is called regular if $E=E(S)$ where $S$ is a regular semigroup.
In particular, $E(M_n(Q))$ is a regular biordered set since the semigroup $M_n(Q)$ is regular.
Recall from \cite{Brittenham2009} that the {\it Graham-Houghton graph} of a
(regular) biordered set $E$ is the bipartite graph  with vertices
the disjoint union of the set of $\gr$-classes of $E$ and the set
of $\gl$-classes of $E$, and with a directed (positively oriented)
edge from an $\gl$-class $L$ to an $\gr$-class $R$ if there is an
idempotent $e \in L \cap R$ (and a corresponding inverse edge from
$R$ to $L$ in this case). One then adds 2-cells to this graph, one
for each singular square $(e,f,g,h)$. Given this square we sew a
2-cell onto this graph with boundary $ef^{-1}gh^{-1}$. The
resulting 2-complex is called the {\it Graham-Houghton complex} of $E$ and we
denote it by $GH(E)$.

The following theorem in \cite{Brittenham2009} is based on the work of
Nambooripad, and is the principal tool used in \cite{Brittenham2009} to
construct maximal subgroups of free idempotent-generated
semigroups on biordered sets.

\begin{thm} \ \cite{Brittenham2009}
\label{thm_fundgroup}
Let $E$ be a regular biordered set.
Then the maximal subgroup of $IG(E)$ containing $e \in E$ is
isomorphic to the fundamental group $\pi_{1}(GH (E), L_{e})$ of
the Graham-Houghton complex of $E$ based at the $\gl$-class $L_{e}$ of $e$.
\end{thm}
The above result shows that the maximal subgroups of $IG(E)$ are determined by relations given by singular squares, when $E$ is a regular biordered set. In fact, the same is true in the case of arbitrary (non-regular) biordered sets, as is shown in \cite{grayta} using Reidemeister--Schreier rewriting methods (see \cite{Magnus1, Ruskuc1999}).

What remains of this section will be dedicated to using Theorem~\ref{thm_fundgroup} to write down a presentation for the group $H(W,IG(E))$. The key task for writing down a presentation for the fundamental group of the connected component of $W$ in this 2-complex is to to make a good choice of spanning tree for the underlying 1-skeleton and so this is what we shall turn our attention to now.

We begin by associating a certain bipartite graph with the Rees structure matrix $P_r$, and then shall define the spanning tree of the connected component of $W$ in the 1-skeleton of the Graham--Houghton complex using a certain
subtree of this bipartite graph.

\begin{defn}\label{def_Delta}
Let $P_{r} = (P_r(Y,X))$ where $Y \in \mathcal{Y}_{r}$, $X\in \mathcal{X}_r$ and $P_r(Y,X) = YX$ if $YX$ is of rank
$r$, and $0$ otherwise. Let $\Delta(P_r)$ denote the bipartite graph with vertex set $\X_r \cup \Y_r$ where $(X,Y)$ is
an edge if and only if $P_r(Y,X)=YX=I_r$.
\end{defn}

Note that there are fewer edges in the graph $\Delta(P_r)$ than there are idempotents in $D_r$, that is, the edges in
this graph just pick out a subset of the idempotents.

The graph $\Delta(P_r)$ is the subgraph of the connected component of $W$ in the 1-skeleton of the Graham--Houghton complex $GH(E)$, with the same vertex set, and edge set corresponding to the positions of the entries $I_r$ in the structure matrix $P_r$.
The use of such bipartite graphs as an approach to the
study of products of elements in Rees matrix semigroups is widespread, see for example \cite{Gr, GGR, GraRuskGraphs,
Hough, HowieGraphs}.

Given $A_1, \ldots, A_k$, a collection of pairwise disjoint subsets of $\{1,\ldots, n\}$, we let $I(A_1|\cdots|A_k)$ denote the $(k \times n)$ matrix with $1$ in positions $(j,a_j)$ $(a_j \in A_j)$ for $1 \leq j \leq k$, and every other entry equal to $0$.
In particular, given a subset $\{i_1, \ldots, i_r \}$ of $\{1, \ldots, n\}$ with $i_1 < \cdots < i_r$, we use $I(i_1|\cdots|i_r)$ to denote the $r \times n$ matrix with $1$ in positions $(j,i_j)$ for $1 \leq j \leq r$, and every other entry equal to $0$. So the $i_1$ to $i_r$th columns of $I(i_1|\cdots|i_r)$ together form a copy of the $r \times r$ identity matrix, and the other columns are all zero vectors.
We call $I(i_1| \cdots | i_r)$ a scattered identity matrix.
Given $m \leq n$ and $\{i_1, \ldots, i_r\} \subseteq \{1,\ldots, m\}$ we use $I_{r \times m}(i_1|\cdots|i_r)$ to denote the $r \times m$ scattered identity matrix with $1$ in positions $(j,i_j)$ for $1 \leq j \leq r$, and every other entry equal to $0$.

Unless otherwise stated, throughout given an $r$-element subset $\{i_1, \ldots, i_r \}$ of $\{1, \ldots, n\}$ we shall adopt the convention that the elements are ordered so that $i_1 < \cdots < i_r$.

We group the vertices of $\Delta(P_r)$ together depending on their leading rows or columns. Given $X \in \X_r$ with
$LR(X) = \{i_1, \ldots, i_r \}$ where $i_1 < \cdots < i_r$ we shall say that $X$ belongs to the region $(i_1 < \cdots <
i_r)$. Similarly, given $Y \in \Y_r$ with $LC(Y) = \{i_1, \ldots, i_r \}$ where $i_1 < \cdots < i_r$ we shall say that
$Y$ belongs to the region $(i_1 < \cdots < i_r)$. Moreover, we let
\[
(i_1 < \cdots < i_r) \times (j_1 < \cdots < j_r)
\]
denote the subgraph of $\Delta(P_r)$ induced by the set of all vertices $X \in \mathcal{X}_r$ belonging to the region
$(i_1 < \cdots < i_r)$ together with all vertices $Y \in \mathcal{Y}_r$ belonging to the region $(j_1 < \cdots < j_r)$.
In particular, we let
\[
\Delta(i_1 < \cdots < i_r) = (i_1 < \cdots < i_r) \times (i_1 < \cdots < i_r),
\]
and call these the diagonal regions.

\begin{sloppypar}
We define a natural order $\preceq$ on the set of $r$-element subsets of $\{1,\ldots, n\}$ where $\{1,\ldots, r\}
\preceq A$ for every $r$-element subset of $\{1,\ldots, n\}$, and given $A = \{a_1, \ldots, a_r \} \neq \{1,\ldots, r
\}$ with $a_1 < a_2 < \cdots < a_r$ we set
\begin{align*}
\{ a_1, \ldots, a_{m-1}, a_{m} - 1, a_{m+1}, \ldots, a_r \}
\preceq
\{ a_1, \ldots, a_{m-1}, a_{m}, a_{m+1}, \ldots, a_r \}
\end{align*}
where $m \in \{1,\ldots, r\}$ is the smallest subscript such that $a_m \neq m$, and then we take the reflexive
transitive closure to obtain the relation  $\preceq$. Clearly this defines a partial order on the $r$-element subsets of
$\{1,\ldots, n\}$ and this order has a unique minimal element $\{1,\ldots,r\}$ which lies below every other element of
the poset. This order clearly induces an order on the diagonal regions of the bipartite graph $\Delta(P_r)$. Note that the poset of $r$-element subsets of $\{1,\ldots, n\}$ under the $\preceq$-relation has the property that
every non-minimal element $p$ of the poset covers exactly one other element. That is, for every non-minimal $p$ there is precisely one element $q$ of the poset such that $q \prec p$ and there is no element $z$ satisfying $q \prec z \prec p$.
In particular, the Hasse diagram of such a poset is a tree; see Figure~\ref{fig_order} for an illustration of this when $n=6$ and $r=2$.
\end{sloppypar}

\begin{figure}[t]
\begin{center}
\begin{tikzpicture}[scale=0.3, yscale=-1]
\tikzstyle{every node} = [circle, fill=gray!30]
\node (12) at (3,19) {\small 12};
\node (13) at (3,16) {\small 13};
\node (14) at (3,13) {\small 14};
\node (15) at (3,10) {\small 15};
\node (16) at (3,7) {\small 16};
\node (23) at (7,15) {\small 23};
\node (24) at (7,12) {\small 24};
\node (25) at (7,9) {\small 25};
\node (26) at (7,6) {\small 26};
\node (34) at (11,11) {\small 34};
\node (35) at (11,8) {\small 35};
\node (36) at (11,5) {\small 36};
\node (45) at (15,7) {\small 45};
\node (46) at (15,4) {\small 46};
\node (56) at (19,3) {\small 56};
\thickedgedraw{12}{13};
\thickedgedraw{13}{14};
\edgedraw{14}{15};
\edgedraw{15}{16};
\thickedgedraw{13}{23};
\edgedraw{14}{24};
\edgedraw{15}{25};
\edgedraw{16}{26};
\edgedraw{24}{34};
\edgedraw{25}{35};
\edgedraw{26}{36};
\edgedraw{35}{45};
\edgedraw{36}{46};
\edgedraw{46}{56};
\end{tikzpicture}
\end{center}

\caption{An example of the $\preceq$ ordering for $n=6$, $r=2$. The bold edges correspond to the regions of the graph $\Delta(P_r)$ drawn in Figure~\ref{fig_regions}.}\label{fig_order}

\end{figure}
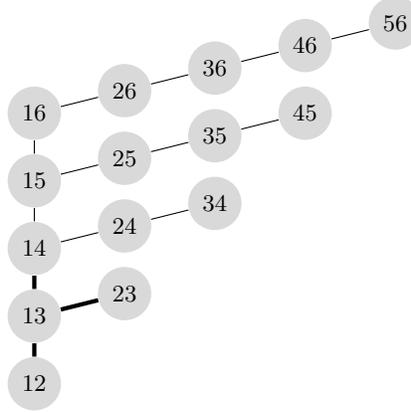

\begin{figure}[t]
\begin{center}
\scalebox{1.0}
{
\begin{tikzpicture}
[scale=0.3,
blackbox/.style={draw, color=gray!30, fill=gray!30, rectangle, minimum height=8.8mm, minimum width=28mm},
dottedbox/.style={draw, rectangle, loosely dashed, minimum height=3.5cm, minimum width=8cm}]
\tikzstyle{vertex}=[circle,draw=black, fill=black, inner sep = 0.3mm]
\node at (6,11) [blackbox] {};
\node at (16,11) [blackbox] {};
\node at (26,11) [blackbox] {};
\node at (36,11) [blackbox] {};
\node at (6,3) [blackbox] {};
\node at (16,3) [blackbox] {};
\node at (26,3) [blackbox] {};
\node at (36,3) [blackbox] {};
\node (a1)  [vertex,label={90:{\tiny $I(1|2)$ \; \;}}] at (3,11) {};
\node (a2)  [vertex,label={90:{\tiny \; \; $I(1|2,3)$}}] at (5,11) {};
\node (a3)  [vertex,label={180:{}}] at (7,11) {};
\node (a4)  [vertex,label={180:{}}] at (9,11) {};
\node (b1)  [vertex,label={90:{\tiny $I(1|3)$ \; \;}}] at (13,11) {};
\node (b2)  [vertex,label={90:{\tiny \; \; $I(1|3,4)$}}] at (15,11) {};
\node (b3)  [vertex,label={180:{}}] at (17,11) {};
\node (b4)  [vertex,label={90:{\tiny $I(1,2|3)$}}] at (19,11) {};
\node (c1)  [vertex,label={90:{\tiny $I(1|4)$}}] at (23,11) {};
\node (c2)  [vertex,label={180:{}}] at (25,11) {};
\node (c3)  [vertex,label={180:{}}] at (27,11) {};
\node (c4)  [vertex,label={180:{}}] at (29,11) {};
\node (d1)  [vertex,label={90:{\tiny $I(2|3)$}}] at (33,11) {};
\node (d2)  [vertex,label={180:{}}] at (35,11) {};
\node (d3)  [vertex,label={90:{\tiny $$}}] at (37,11) {};
\node (d4)  [vertex,label={180:{}}] at (39,11) {};
\node (a1')  [vertex,label={270:{\tiny \; $I(1|2)^T$}}] at (3,3) {};
\node (a2')  [vertex,label={180:{}}] at (5,3) {};
\node (a3')  [vertex,label={180:{}}] at (7,3) {};
\node (a4')  [vertex,label={180:{}}] at (9,3) {};
\node (b1')  [vertex,label={270:{\tiny \; $I(1|3)^T$}}] at (13,3) {};
\node (b2')  [vertex,label={180:{}}] at (15,3) {};
\node (b3')  [vertex,label={180:{}}] at (17,3) {};
\node (b4')  [vertex,label={180:{}}] at (19,3) {};
\node (c1')  [vertex,label={270:{\tiny \; $I(1|4)^T$}}] at (23,3) {};
\node (c2')  [vertex,label={180:{}}] at (25,3) {};
\node (c3')  [vertex,label={180:{}}] at (27,3) {};
\node (c4')  [vertex,label={180:{}}] at (29,3) {};
\node (d1')  [vertex,label={270:{\tiny \; $I(2|3)^T$}}] at (33,3) {};
\node (d2')  [vertex,label={180:{}}] at (35,3) {};
\node (d3')  [vertex,label={180:{}}] at (37,3) {};
\node (d4')  [vertex,label={180:{}}] at (39,3) {};
\edgedraw{a1'}{a1};
\edgedraw{a1'}{a2};
\edgedraw{a1'}{a3};
\edgedraw{a1'}{a4};
\edgedraw{b1'}{b1};
\edgedraw{b1'}{b2};
\edgedraw{b1'}{b3};
\edgedraw{b1'}{b4};
\edgedraw{c1'}{c1};
\edgedraw{c1'}{c2};
\edgedraw{c1'}{c3};
\edgedraw{c1'}{c4};
\edgedraw{d1'}{d1};
\edgedraw{d1'}{d2};
\edgedraw{d1'}{d3};
\edgedraw{d1'}{d4};
\dashedgedraw{a1'}{a1};
\dashedgedraw{a2'}{a1};
\dashedgedraw{a3'}{a1};
\dashedgedraw{a4'}{a1};
\dashedgedraw{b1'}{b1};
\dashedgedraw{b2'}{b1};
\dashedgedraw{b3'}{b1};
\dashedgedraw{b4'}{b1};
\dashedgedraw{c1'}{c1};
\dashedgedraw{c2'}{c1};
\dashedgedraw{c3'}{c1};
\dashedgedraw{c4'}{c1};
\dashedgedraw{d1'}{d1};
\dashedgedraw{d2'}{d1};
\dashedgedraw{d3'}{d1};
\dashedgedraw{d4'}{d1};
\thickedgedraw{b1'}{a2};
\thickedgedraw{c1'}{b2};
\thickedgedraw{d1'}{b4};
\node (dots) at (42,11) {$\ldots$};
\node (dots2) at (42,3) {$\ldots$};
\node (12) at (6,14) {\tiny $(1 < 2)$};
\node (13) at (16,14) {\tiny $(1 < 3)$};
\node (14) at (26,14) {\tiny $(1 < 4)$};
\node (23) at (36,14) {\tiny $(2 < 3)$};
\node (12) at (6,0) {\tiny $(1 < 2)$};
\node (13) at (16,0) {\tiny $(1 < 3)$};
\node (14) at (26,0) {\tiny $(1 < 4)$};
\node (23) at (36,0) {\tiny $(2 < 3)$};
\end{tikzpicture}
}
\end{center}
\caption{A partial view of the graph $\Delta(P_r)$, where $n=6$ and $r=2$, with edges from the spanning tree $T_{n,r}$ indicated.
The bold lines represent edges of type (T3); the dashed lines those of type (T2); while the remaining edges are those of type (T1). Note that the edges $(I(i|j),I(i|j)^T)$ are both of type (T1) and (T2).
When the spanning tree is quotiented out by the diagonal regions we obtain a graph that is isomorphic to the Hasse graph illustrated in Figure~\ref{fig_order}. In particular, the three bold edges between regions in this figure correspond in the obvious natural way to the three bold edges in Figure~\ref{fig_order}.
}\label{fig_regions}
\end{figure}
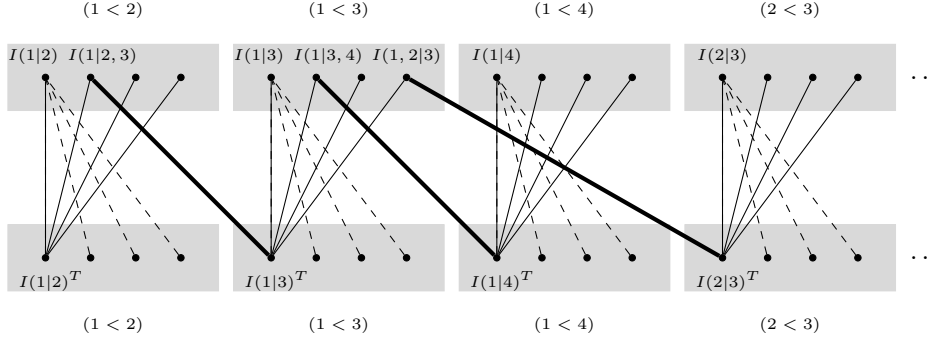

Let $T_{n,r}$ be the subgraph of $\Delta(P_r)$ spanned by the edges:
\begin{enumerate}[(T1)]
\item $(I(i_1|\cdots|i_r)^T, Y)$ where $Y \in \Y_r$ belongs to the region $(i_1 < \cdots < i_r)$;
\item $(X, I(i_1|\cdots|i_r))$ where $X \in \X_r$ belongs to the region $(i_1 < \cdots < i_r)$; and
\item $(I(i_1|\cdots|i_r)^T, I(i_1|\cdots|i_{j-1}|i_j-1, i_j|i_{j+1}|\cdots|i_r))$ where $i_1 = 1$, $i_2 = 2,$ $\ldots,$ $i_{j-1} = j-1$ but $i_j \neq j$,
\end{enumerate}
where $\{i_1, \ldots, i_r \}$ ranges through all $r$-element subsets of $\{1,\ldots, n\}$.

If one just takes the edges (T1) and (T2) one obtains a bipartite graph whose connected components are connected subgraphs of the regions
$\Delta(i_1 < \cdots < i_r)$. The remaining edges (T3) give exactly one edge connecting every pair of regions that are
adjacent under the $\preceq$ order. In particular this means that the graph obtained by factoring out the spanning tree by the
equivalence relation given by the diagonal regions is isomorphic to the Hasse diagram of the poset of $r$-element subsets of
$\{1,\ldots, n\}$ under $\preceq$.

Each of (T1), (T2) and (T3) is easily seen to define a subset of the edges of $\Delta(P_r)$. Moreover, using the observations made above it
is easily verified that $T_{n,r}$ is a spanning tree for the graph $\Delta(P_r)$.
An illustration of the spanning tree $T_{n,r} = T_{6,2}$ in the graph $\Delta(P_r) = \Delta(P_2)$ is given in Figure~\ref{fig_regions}.
In fact, not only is $T_{n,r}$ a spanning tree of $\Delta(P_r)$, but it is also easily seen to be a spanning tree of the 1-skeleton of the connected component of $W$ in the Graham--Houghton complex $GH(E)$. We use this fact below where we write down a presentation for the group $H(W,IG(E))$.

In light of Theorem~\ref{thm_fundgroup} we would now like to characterise the singular squares in the $\gd$-class $D_r$.
It is easy to show in general that if $(e,f,g,h)$ is a singular square then $\{e,f,g,h\}$ forms a $2 \times 2$ rectangular band. In other words, for a square to stand a chance of being singular it must be a rectangular band. In \cite{margolismeakinip} it is shown that in the full linear semigroup the converse is also true.
\begin{thm} \ \cite[Theorem~4.3]{margolismeakinip} \label{rectsing}
Every non-trivial rectangular band in $M_{n}(\F)$ is a singular square.
\end{thm}
Our interest is in the set of singular squares in the $\gd$-class $D_r$. Such a singular square may be given either by listing the four idempotents that make up that square, or by giving a quadruple $(X,X',Y,Y') \in \X_r \times \X_r \times \Y_r \times \Y_r$ which specifies the coordinates of a singular square $(e_{X,Y}, e_{X,Y'}, e_{X',Y'}, e_{X',Y})$ of idempotents. Since it will always be clear from context what we mean, we shall also call such quadruples $(X,X',Y,Y')$ singular squares.

Let $\Sigma \subseteq \X_r \times \X_r \times \Y_r \times \Y_r$ be the set of all singular squares of $D_r$, which by the above result correspond precisely the set of rectangular bands in $D_r$. In terms of the structure matrix $P_r$ it is easily verified that $(X,X',Y,Y')$ corresponds to a rectangular band if and only if the equality
\begin{equation}
P_r(Y,X) P_r(Y',X)^{-1} = P_r(Y,X') P_r(Y',X')^{-1}
\label{eq_99}
\end{equation}
holds in the group $GL_r(Q)$. (Actually this is a general fact describing $2 \times 2$ rectangular bands in Rees matrix
semigroups.)

With the above notation, it now follows from the definition of the Graham--Houghton complex along with Theorem~\ref{thm_fundgroup} that the maximal subgroup $H(W,IG(E))$ is defined
by the presentation with generators
\begin{align}
\label{eqn_generators}
\mathcal{F}=\{ f_{X,Y} : X\in \X_r,\ Y\in \Y_r,\ P_r(Y,X) \neq 0 \},
\end{align}
and defining relations
\begin{alignat}{2}
\label{eqn_middle}
& f_{X,Y}= 1 &  & \quad (X,Y) \in T_{n,r},  \\
\label{eqn_bottom}
& f_{X,Y}^{-1}f_{X,Y'}=f_{X',Y}^{-1}f_{X',Y'}   & & \quad ((X,X',Y,Y')\in\Sigma).
\end{alignat}
Let us denote this presentation by $\mathcal{P}_{r,n}$.

The rest of the paper will be devoted to the proof that when $r < n/3$ this presentation actually defines the general linear group $GL_r(Q)$.

\section{Outline of the Proof}
\label{sec_outline}

The basic idea behind the proof is as follows. We consider two matrices both with rows indexed by $\X_r$ and columns
indexed by $\Y_r$. The first matrix is the transpose $P_r^T$ of the Rees structure matrix of $D_r^0$ defined in
Theorem~\ref{thm_rees}. The second is the $\X_r \times \Y_r$ matrix with non-zero entries the abstract generators
$f_{X,Y}$ from the presentation $\mathcal{P}_{r,n}$.
So, we view the set of generators $\mathcal{F}$ given in \eqref{eqn_generators} as being arranged in a matrix in a natural way
where the entry indexed by the pair $(X,Y)$ is equal to the generator $f_{X,Y}$ and all the other entries are set to
$0$. We shall carry out a sequence of Tietze transformations to the presentation $\mathcal{P}_{r,n}$ transforming it
into a presentation for the general linear group $GL_r(Q)$. One of the key ideas is that we imagine the two $\X_r
\times \Y_r$ matrices above laid out side-by-side, and then the Rees structure matrix $P_r^T$ acts as a ``guide''
pointing out relations between the generators $f_{X,Y}$ one should be aiming to show hold. The fact that the structure
of $P_r$ influences the relations we obtain should not come as a surprise since, firstly,
the spanning tree $T_{n,r}$ has been defined in terms of $P_r$, which links entries in $P_r$ with the relations
\eqref{eqn_middle}, and secondly, because all the rectangular bands are singular, the relations \eqref{eqn_bottom} in
the presentation correspond exactly to the singular squares which are seen inside $P_r$.

The proof breaks down into the following three main steps.

\

\noindent Stage~1: \emph{Generators $f_{X,Y}$ such that $YX=I_r$.}

\

\noindent In Section~\ref{sec_Ilabels} we prove that for every such generator the relation $f_{X,Y}=1$ is a consequence of the relations in the presentation $\mathcal{P}_{r,n}$ (see Lemma \ref{lem_hash}). This is done by beginning with the relations \eqref{eqn_middle} which tell us that the result holds for every generator $f_{X,Y}$ where the corresponding edge $(X,Y)$ belongs to the spanning tree $T_{n,r}$ and then
extending this, using the relations \eqref{eqn_bottom}, to arbitrary edges from the bipartite graph $\Delta(P_r)$.

\

\noindent Stage~2: \emph{Pairs of generators $f_{A,B}$ and $f_{X,Y}$ such that $BA = YX$.}

\

\noindent In Sections~\ref{sec_multtables} and \ref{sec_lambdalabels} we prove that for every such pair, the relation $f_{A,B} = f_{X,Y}$ is a consequence of the relations in
the presentation $\mathcal{P}_{r,n}$ (see Lemmas \ref{lem_topleft} and \ref{lem_fulltable}). This is achieved in the following way. We fix some element $K \in GL_r(Q)$ and consider all the generators $f_{A,B}$ such that $BA=K$.
Given such a pair of generators $f_{A,B}$, $f_{A',B'}$ we say that they are strongly connected if (1) they are in the same row or column (i.e. $A=A'$ or $B=B'$) and (2) the pair $f_{A,B}$, $f_{A',B'}$ completes to a singular square such that the other pair $f_i$, $f_j$ of the square are both known to satisfy $f_i=f_j=1$ as a consequence of Stage 1. Then a sequence of generators $f_{A,B}$ all satisfying $BA=K$, and such that adjacent terms in the sequence are strongly connected, is called a strong path. In this language, in this step of the proof we prove that for every $K \in GL_r(Q)$, and for every pair $f_{A,B}$, $f_{A',B'}$, if $BA=B'A'=K$ then there is a strong path from $f_{A,B}$ to $f_{A',B'}$. Keeping in mind the relations \eqref{eqn_bottom}, this will suffice to
show that $f_{A,B} = f_{A',B'}$ is a consequence of the relations from the presentation $\mathcal{P}_{r,n}$.

\

\noindent Stage~3: \emph{Defining relations for $GL_r(Q)$}.

\

\noindent By this stage we have transformed $\mathcal{P}_{r,n}$ into a presentation whose generators are in natural one to one
correspondence with the elements of $GL_r(Q)$. Using this correspondence, we denote the generating symbols in this new
presentation by $f_{A}$ where $A \in GL_r(Q)$.
As a consequence of \eqref{eq_99} and the relations \eqref{eqn_middle}, \eqref{eqn_bottom} it follows that the map which sends $f_A$ to $A^{-1} \in GL_r(Q)$ extends to define a well-defined homomorphism from $H(W,IG(E))$ onto $GL_r(Q)$, and this homomorphism maps the generators $f_A$ bijectively to $GL_r(Q)$.
The fact that $f_A$ is mapped to $A^{-1}$ rather than $A$ here reflects the relationship between \eqref{eq_99} and \eqref{eqn_bottom}.
In Section~\ref{sec_uncovering}, we show that for
any pair of matrices $A$, $B$ from $GL_r(Q)$ the relation $f_B f_A = f_{AB}$ belongs to \eqref{eqn_middle} (see Lemma
\ref{lem_multiplicationtable}) and it follows that every word over the generators $f_A$ is actually equal to one of the generators.
Again the form of the relation $f_B f_A = f_{AB}$ comes from the fact that $f_A$ corresponds to $A^{-1}$ and the fact that $B^{-1} A^{-1} = (AB)^{-1}$ in $GL_r(Q)$.
We may then conclude that the \emph{elements} of $H(W,IG(E))$ are in bijective correspondence with the generators $f_A$ and
thus
that the homomorphism from $H(W,IG(E))$ onto $GL_r(Q)$ described above is in fact an isomorphism.

\

Let us conclude this section by making a comment about our method of proof.
We recall that, in general, when writing down a Rees matrix representation for a completely $0$-simple semigroup the structure matrix $P$ is in no sense unique; see \cite[Theorem~3.4.1]{howie95}. Given an arbitrary completely $0$-simple semigroup, and some Rees matrix representation for it, the graph $\Delta(P)$ need not be connected (in fact it will not contain any edges at all if the structure matrix does not contain any $1$s) even if the semigroup is idempotent generated.
On the other hand, it is always possible to normalise the matrix putting it into a particular form, introduced in
\cite{Gr} and later utilised in \cite{GGR}, called Graham normal form. When the corresponding completely $0$-simple
semigroup is idempotent generated, if the structure matrix $P$ is in Graham normal form, then $\Delta(P)$ will be
connected. It just so happens that the Rees matrix representation we work with in this paper is in Graham normal form,
and the decision to work with this particular Rees matrix representation is an important part of the proof, since we
need the graph $\Delta(P_r)$ to be connected in order to find the spanning tree $T_{n,r}$ which is the starting point
of the proof of the main theorem. This suggests that putting the structure matrix into Graham normal form would be a
sensible first step when investigating maximal subgroups of free idempotent generated semigroups in general.

\section{Generators $f_{X,Y}$ such that $YX=I_r$}
\label{sec_Ilabels}

In this section we work through Stage~1 of the proof of the main theorem. We continue using the notation and
definitions introduced in Section~\ref{sec_preliminaries}. So in particular, $P_r$ denotes the structure matrix for the
Rees matrix representation for the principal factor $D_r^0$ of $M_n(Q)$, $\Delta(P_r)$ is the bipartite graph defined
in Definition~\ref{def_Delta} whose edges correspond to occurrences of $I_r$ in $P_r$, and $T_{n,r}$ is the spanning tree
of $\Delta(P_r)$ spanned by the edges (T1)--(T3). Recall that our aim is to show that for every edge $(X,Y)$ of
$\Delta(P_r)$ the relation $f_{X,Y}=1$ is a consequence of the presentation $\mathcal{P}_{r,n}$.
Even though in the statement of the main result Theorem~\ref{thm_main} we insist that $r < n/3$, all of the results in this section will be proved under the weaker assumption that $1 \leq r < n-1$ and this assumption about the relationship between $n$ and $r$ will remain in place throughout the section.

From the relations \eqref{eqn_middle}
we
already know $f_{X,Y}=1$ for every edge $(X,Y)$ in the spanning tree $T_{n,r}$. With this initial information, together
with the relations \eqref{eqn_bottom} from the presentation, we shall complete the proof of Stage~1 by proving the
following result.

\begin{lem}\label{lem_hash}
Let $n$ and $r$ be positive integers with $1 \leq r < n-1$, and
let $Y\in\Y_r$, $X\in\X_r$ such that $YX=I_r$. Then the relation $f_{X,Y}=1$ is a consequence of the relations
\eqref{eqn_middle}--\eqref{eqn_bottom}.
\end{lem}

It will be useful to rephrase the problem in purely combinatorial terms. Let $n$ and $r$ be positive integers with $1 \leq r < n-1$, and let $\Delta_{n,r}$ be the bipartite graph given in Definition~\ref{def_Delta}.
Now we shall colour the edges of $\Delta_{n,r}$ so that every edge is either red or blue. Initially we colour all edges
from the spanning tree $T_{n,r}$ of $\Delta_{n,r}$ blue (i.e. the edges (T1), (T2) and (T3)) and all the other edges
red. Our aim is to turn the colour of every edge from red to blue, in the following way. For every square of edges in
$\Delta_{n,r}$ if three of the edges of the square are blue, and the fourth is red, the we can transform the fourth
edge from red to blue. We call such a transformation an \emph{elementary edge colour transformation}. The remainder of
this section will be concerned with proving the following result.

\begin{prop}
\label{lem_bipartitecolours}
Let $n$ and $r$ be positive integers with $1 \leq r < n-1$, and let $\Delta_{n,r}$ be the coloured bipartite graph defined above with blue edges for every edge in the spanning tree
$T_{n,r}$ and all other edges coloured red. Then every red edge of $\Delta_{n,r}$ may be turned blue by a finite sequence of elementary edge colour transformations.
\end{prop}

Before proving Proposition \ref{lem_bipartitecolours} we now show how Lemma \ref{lem_hash} follows from it.

\begin{proof}[Proof of Lemma \ref{lem_hash}]
If $(X,Y)\in T_{n,r}$ then we are done by \eqref{eqn_middle}. By Proposition \ref{lem_bipartitecolours} every
edge $(X,Y)$ can be reached, and turned blue, by a finite sequence of elementary edge colour transformations. The proof now
proceeds by induction on the number of elementary edge colour transformations required to turn the edge $(X,Y)$ blue. For
the inductive step we have vertices $X'\in X_r$ and $Y'\in\Y_r$ such that the edges $(X',Y)$, $(X,Y')$ and $(X',Y')$
have all been turned blue, and by induction $f_{X',Y}=1$, $f_{X,Y'}=1$ and $f_{X',Y'}=1$ are all consequences of the
relations \eqref{eqn_middle}--\eqref{eqn_bottom}. Since by definition of $\Delta_{n,r}$,
$$P_r(Y,X)=P_r(Y,X')=P_r(Y',X)=P_r(Y',X')=1$$
it follows by Theorem \ref{rectsing} and \eqref{eq_99} that this square is singular. Hence, by applying the relation
\eqref{eqn_bottom} from the presentation $\mathcal{P}_{n,r}$ we deduce $f_{X,Y}=1$.
\end{proof}
After first proving some general lemmas we shall then deduce that Proposition \ref{lem_bipartitecolours} holds for all
pairs $(n,r)=(n,1)$ with $n \geq 3$. Then we prove the result for an arbitrary pair $(n,r)$ (with $1 < r < n-1$) where we may (and
shall) assume inductively that the result holds for the pair $(n-1,r-1)$ (which we note still satisfies $r-1 <
(n-1) - 1$). Induction will be applied by finding a natural copy of the coloured bipartite graph $\Delta_{n-1,r-1}$ as an
induced subgraph of $\Delta_{n,r}$.

\begin{lem}
\label{lem_useful} Let $(X,Y), (X',Y)$ and $(X,Y')$ be edges of $\Delta_{n,r}$.
\begin{enumerate}
\item[(i)]
If $X$ and $X'$ are in the same region, and $(X,Y)$ has been turned blue, then
$(X',Y)$ can be turned blue.
\item[(ii)]
If $Y$ and $Y'$ are in the same region, and $(X,Y)$ has been turned blue, then
$(X,Y')$ can be turned blue.
\end{enumerate}
\end{lem}
\begin{proof}
(i) Suppose that $X,X' \in (i_1 < i_2 < \cdots i_r)$. Then both of the edges
\[
(X, I(i_1|i_2|\cdots|i_r)) \ \ \mbox{and} \ \ (X', I(i_1|i_2|\cdots|i_r)),
\]
belong to the tree $T_{n,r}$ and thus are assumed already to be blue edges.
Now the subgraph induced by these two edges together with the edges $(X,Y)$ and $(X',Y)$ form a square in
$\Delta_{n,r}$. It is then immediate from the definition of elementary edge colour transformation that once $(X,Y)$ has
been turned blue $(X',Y)$ may also be turned blue.

(ii) The proof is dual to that of (i), making use of the fact that,
with $Y, Y' \in (i_1 < i_2 < \cdots i_r)$,
both of the edges
\[
(I(i_1|i_2|\cdots|i_r)^T,Y) \ \ \mbox{and} \ \ (I(i_1|i_2|\cdots|i_r)^T,Y'),
\]
belong to the spanning tree $T_{n,r}$.
\end{proof}

\begin{lem}
\label{lem_connectedregions} Let $\Gamma$ be the subgraph
\[
(i_1 < i_2 < \cdots < i_r) \times (j_1 < j_2 < \cdots < j_r),
\]
of $\Delta_{n,r}$, and suppose that all the edges of $\Gamma$ belong to a single connected component of $\Gamma$.
Then every edge of $\Gamma$ may be turned blue provided at least one edge of $\Gamma$ has been turned blue.
\end{lem}
\begin{proof}
Let $(X,Y)$ and $(X',Y')$ be edges in $\Gamma$.
Suppose that the edges $(X,Y)$ and $(X',Y')$ are adjacent i.e. that $X=X'$ or $Y=Y'$. Then
it follows from Lemma~\ref{lem_useful} that $(X,Y)$ can be turned blue if and only if $(X',Y')$ can be turned blue. But
since all the edges of $\Gamma$ belong to a single connected component
there is a sequence of edges between $(X,Y)$ and $(X',Y')$ where adjacent edges
in the sequence are adjacent in $\Gamma$. The result is now immediate.
\end{proof}
It should be noted that it is not true in general that every region will satisfy the hypotheses of Lemma~\ref{lem_connectedregions}.

\begin{lem}
\label{lem_finishingoff} Let $(X,Y)$ be an edge of $\Delta_{n,r}$. For any subset $\{i_1, \ldots, i_r\}$ of
$\{1,\ldots,n\}$, if $(X,Y)$ belongs to any of the following subgraphs:
\begin{enumerate}
\item[(i)]
$(i_1 < i_2 < \cdots < i_r)
\times
(i_1 < i_2 < \cdots < i_r)$,
\item[(ii)]
$(i_1 < i_2 < \cdots < i_r)
\times
(i_1-1 < i_2 < \cdots < i_r)$ $(i_1 \geq 2)$, or
\item[(iii)]
$(i_1-1 < i_2 < \cdots < i_r)
\times
(i_1 < i_2 < \cdots < i_r)$ $(i_1 \geq 2)$,
\end{enumerate}
then $(X,Y)$ can be turned blue.
\end{lem}
\begin{proof}
Suppose that $(X,Y)$ belongs to the subgraph (i) and let $I_0 = I(i_1|\cdots|i_r)$. Then the edges $(I_0^T,I_0)$,
$(X,I_0)$ and $(I_0^T,Y)$, which all belong to the tree $T_{n,r}$ and hence are blue edges, together with the edge $(X,Y)$
form a square, and hence $(X,Y)$ can be turned blue.

Now suppose that $(X,Y)$ belongs to the subgraph (ii) which we shall denote here by $\Gamma$. We claim that all the
edges of $\Gamma$ belong to a single connected component of $\Gamma$. Indeed, for every vertex $X \in (i_1 < i_2 <
\cdots < i_r)$, the edge
\begin{align}
\label{eqn_oneofthese}
(X, I(i_1-1, i_1 | i_2 | \cdots | i_r))
\end{align}
belongs to $\Gamma$ since $I(i_1-1, i_1 | i_2 | \cdots | i_r) X = I_r$ and, since $\Gamma$ is bipartite, every other
edge of $\Gamma$ shares a common vertex with an edge from of the form \eqref{eqn_oneofthese}. But the edge
$(I(i_1|i_2|\cdots|i_r)^T, I(i_1-1, i_1 | i_2 | \cdots | i_r))$ from the spanning tree $T_{n,r}$ belongs to (ii), and is
blue by assumption, and so it follows from Lemma~\ref{lem_connectedregions} that, since one edge has been turned blue
and all the edges in the region belong to a single connected component, every edge in the subgraph (ii) can be turned
blue.

Finally suppose that $(X,Y)$ belongs to the subgraph (iii). This is the most difficult of the three cases since there
are no edges from the spanning tree $T_{n,r}$ in this region. By a dual argument to case (ii), since the mapping $X \mapsto
X^T$ is an automorphism of the graph $\Delta_{n,r}$ preserving regions, we conclude that the subgraph (iii) has a
connected component that contains all of its edges, and thus, by Lemma~\ref{lem_connectedregions} it will suffice to
show that at least one edge in each subgraph of type (iii) can be turned blue.

We treat the case $i_1-1=1$ (that is, the case $\{i_1-1,i_1\}=\{1,2\}$) separately.

\

\noindent \textit{Case~1: $i_1 \geq 3$.}
First observe that the edge
\begin{align}\label{edge1}
(
I(i_1 | i_2 | \ldots | i_r)^T,
I(i_1-2, i_1-1, i_1 | i_2 | \ldots | i_r)
)
\end{align}
completes to a square in $\Delta_{n,r}$ when taken together with the following three edges
\begin{align*}
(I(i_1| i_2 | \cdots | i_r)^T,
I(i_1-1, i_1 | i_2 | \cdots | i_r)), \\
(I(i_1-1 | i_2 | \cdots | i_r)^T,
I(i_1-2, i_1-1, i_1 | i_2 | \cdots | i_r)), \\
(I(i_1-1 | i_2 | \cdots | i_r)^T,
I(i_1-1, i_1 | i_2 | \cdots | i_r)).
\end{align*}
But these three edges are from subgraphs of type (ii), (ii) and (i), respectively, and therefore by parts (i) and (ii)
we know that all three of these edges may be turned blue, and therefore the edge \eqref{edge1} may be turned blue.
Next consider the edge
\begin{align}\label{edge2}
(I(i_1 | i_2 | \cdots | i_r)^T,
I(i_1-2, i_1 | i_2 | \cdots | i_r)).
\end{align}
Comparing the edges \eqref{edge1} and \eqref{edge2}, since $I(i_1-2, i_1-1, i_1 | i_2 | \cdots | i_r)$ and
$I(i_1-2, i_1 | i_2 | \cdots | i_r)$ both belong to the region
\[
(i_1-2 < i_2 < i_3 < \cdots < i_r ),
\]
and since the edge \eqref{edge1} has been turned blue, it follows from Lemma~\ref{lem_useful}(ii) that
we can turn the edge \eqref{edge2} blue.

Finally, consider the edge
\begin{align}\label{eqn_offedge}
(
I(i_1-1, i_1 | i_2 | \ldots | i_r)^T,
I(i_1 | i_2 | \ldots | i_r)
).
\end{align}
This edge completes to a square in $\Delta_{n,r}$ when taken together with the following three edges
\begin{align*}
(I(i_1-1, i_1 | i_2 | \ldots | i_r)^T,
I(i_1-2, i_1 | i_2 | \ldots | i_r)), \\
(I(i_1 | i_2 | \cdots | i_r)^T,
I(i_1 | i_2 | \cdots | i_r)), \\
(I(i_1 | i_2 | \cdots | i_r)^T,
I(i_1-2, i_1 | i_2 | \cdots | i_r)).
\end{align*}
The last of these three edges is \eqref{edge2} which we have already shown can be turned blue, while the
other two edges belong to subgraphs of types (ii) and (i), respectively, so are also blue. Thus we deduce that the edge
\eqref{eqn_offedge} can be turned blue and, since this edge belongs to the subgraph (iii), this completes the proof in
this case.

\

\noindent \textit{Case~2: $\{i_1-1,i_1\}=\{1,2\}$.}  In this case,
choose $k \in \{1,\ldots,
n\} \setminus \{1,2,i_2, i_3, \ldots, i_r \}$. This is possible since $n>r+1$. Then the subgraph of
$\Delta_{n,r}$ induced by the four vertices
\begin{align*}
\Y_r: & & I(1,2|i_2|i_3|\cdots|i_r) & & & I(2,k|i_2|i_3|\cdots|i_r) \\
\X_r: & & I(1,k|i_2|i_3|\cdots|i_r)^T & & & I(2|i_2|i_3|\cdots|i_r)^T
\end{align*}
is a square, three of whose edges belong to subgraphs of types (i) or (ii) and so are blue edges, which means that the
fourth edge
\[
(
I(1,k|i_2|i_3|\cdots|i_r)^T,
I(2,k|i_2|i_3|\cdots|i_r)
),
\]
which belongs to the subgraph (iii), can be turned blue, completing the proof for this case. (Note that these four
matrices are in RRE form regardless of the value of $k$.) \end{proof}
We give another general result.

\begin{lem}
\label{lem_closer} Let $(X,Y)$ be an edge of $\Delta_{n,r}$. If $(X,Y)$ belongs to the subgraph
\begin{align}
\label{eqn_similar}
(i_1 < i_2 < \cdots < i_r)
\times
(j_1 < i_2 < \cdots < i_r),
\end{align}
then the edge $(X,Y)$ can be turned blue.
\end{lem}
\begin{proof}
We claim that this subgraph has a single connected component containing all of its edges. When $i_1 = j_1$, so the
subgraph is a diagonal region, this is immediate from the definition of the spanning tree $T_{n,r}$. Now suppose $i_1 < j_1 <
i_2$, the other case being dual. Then for every vertex $B$ in $(j_1 < i_2 < \cdots < i_r)$ the edge
\begin{align}
\label{eqn_frog}
(
I(i_1, j_1|i_2|i_3|\cdots|i_r)^T,
B
)
\end{align}
belongs to $\Delta_{n,r}$ and, since the graph induced by this region is bipartite, it follows that every edge in this
subgraph shares a vertex with one of the edges \eqref{eqn_frog}, proving the claim. Thus by
Lemma~\ref{lem_connectedregions} for each subgraph of the form \eqref{eqn_similar} once we have turned one edge blue,
we can conclude that every other edge in that subgraph can be turned blue.

We prove the lemma by induction on $|i_1-j_1|$. When $|i_1-j_1| \leq 1$ the result holds by
Lemma~\ref{lem_finishingoff}, so suppose otherwise and assume that the result holds for smaller values of $|i_1-j_1|$.
Suppose that $i_1<j_1$ and $j_1 \neq i_1+1$, the other case is dual. Now consider the subgraph of $\Delta_{n,r}$
induced by the four vertices
\begin{align*}
\Y_r: & & I(j_1|i_2|\cdots|i_r) & & & I(i_1, i_1+1|i_2|\cdots|i_r) \\
\X_r: & & I(i_1,j_1|i_2|\cdots|i_r)^T & & & I(i_1+1, j_1|i_2|i_3|\cdots|i_r)^T.
\end{align*}
Three of these edges belong to subgraphs whose edges may be turned blue by induction, thus the remaining edge
\[
(
I(i_1,j_1|i_2|\cdots|i_r)^T,
I(j_1|i_2|\cdots|i_r)
),
\]
which belongs to the subgraph
\[
(i_1 < i_2 < \cdots < i_r) \times
(j_1 < i_2 < \cdots < i_r)
\]
may also be turned blue, completing the inductive step, and hence the proof of the lemma.
\end{proof}
The following result will serve as a family of base cases for the induction proving
Proposition~\ref{lem_bipartitecolours}.

\begin{cor}
\label{cor_basecases} For every positive integer $n \geq 3$, every edge $(X,Y)$ in $\Delta_{n,1}$ can be turned blue.
\end{cor}
\begin{proof}
In this case every edge of $\Delta_{n,1}$ belongs to a region of the form \eqref{eqn_similar} and hence can be turned blue by Lemma~\ref{lem_closer}.
\end{proof}
So, from now on in this section we may suppose that
$r$ and $n$ are integers satisfying
$1< r < n-1$, shall assume that
Proposition~\ref{lem_bipartitecolours} holds for the pair $(n-1,r-1)$, and then prove under this assumption that the
proposition holds for the pair $(n,r)$. Then by induction, with Corollary~\ref{cor_basecases} dealing with the base cases,
this will suffice to prove Proposition~\ref{lem_bipartitecolours}. These assumptions will remain in place for the rest
of this section.

In order to apply our inductive assumption we shall first need to identify a natural subgraph of $\Delta_{n,r}$ which is isomorphic to $\Delta_{n-1,r-1}$. Let $\Delta_{n,r}'$ denote the subgraph of $\Delta_{n,r}$ induced by the set $\Y_r'$ of all vertices from $\Y_r$ of the form
\[
Y = \left[
\begin{array}{cc}
\begin{matrix}
1
\end{matrix}\;\vline
&
\begin{matrix}
0 & 0 & 0 & \cdots & 0 & 0
\end{matrix}
\\\hline
\begin{matrix}
0 \\\vdots \\ 0
\end{matrix}\;\vline
&
\begin{matrix}
\widehat{Y}
\end{matrix}
\end{array}\right],
\]
together with the set $\mathcal{X}_r' \subseteq \mathcal{X}_r$ of transposes of the elements of $\Y_r'$.
Since $Y \in \Y_r' \subseteq \Y_r$ is an $r \times n$ rank $r$ matrix in RRE form, it follows that $\widehat{Y} \in
\Y_{n-1,r-1}$, where $\Y_{n-1,r-1}$ denotes the set of all $(r-1) \times (n-1)$ rank $r-1$ matrices in RRE form.
Conversely given any $(r-1) \times (n-1)$ rank $r-1$ matrix $\widehat{Y}$ in RRE form, the matrix $Y$ above is clearly
then an $r \times n$ rank $r$ matrix in RRE form. Thus $\; \widehat{} \; $ defines, in a natural way, a bijection
between the subset $\Y_r'$ of $\Y_r$ and the set $\Y_{n-1,r-1}$. The obvious dual statements hold for pairs $X \in
\X_r$, $X' \in \X_{n-1,r-1}$ where $\X_{n-1,r-1}$
denotes the set of transposes of elements of $\Y_{n-1,r-1}$. Therefore we have a natural bijection
\begin{align}
\label{eqn_naturalbij}
\widehat{\;} \; :(\X_r' \times \Y_r') \rightarrow (\X_{n-1,r-1} \times \Y_{n-1,r-1}), \quad (X,Y) \mapsto (\widehat{X},\widehat{Y}).
\end{align}

Next we observe that the bijection \eqref{eqn_naturalbij} is actually an isomorphism between the subgraph $\Delta_{n,r}'$ of
$\Delta_{n,r}$ induced by $\X_r' \cup \Y_r'$ and the graph $\Delta_{n-1,r-1}$. This is easily seen. Indeed, for every
pair $(X,Y) \in \X_r' \times \Y_r'$ we have
\[
YX = \left[
\begin{array}{cc}
\begin{matrix}
1
\end{matrix}\;\vline
&
\begin{matrix}
0 & \cdots & 0
\end{matrix}
\\\hline
\begin{matrix}
0 \\\vdots \\ 0
\end{matrix}\;\vline
&
\begin{matrix}
\widehat{Y}\widehat{X}
\end{matrix}
\end{array}
\right] \in M_r(\F),
\]
in particular $YX = I_r$ if and only if $\widehat{Y}\widehat{X}=I_{r-1}$, and therefore $(X,Y)$ is an edge of
$\Delta_{n,r}'$ if and only if $(\widehat{X},\widehat{Y})$ is an edge of $\Delta_{n-1,r-1}$.

Finally, looking at the list of edges (T1), (T2) and (T3) in the definition of the spanning tree, we see that the
edge $(X,Y)$ belongs the spanning tree $T_{n,r}$ of $\Delta_{n,r}$ if and only if the edge $(\widehat{X},\widehat{Y})$
belongs to the spanning tree $T_{n-1,r-1}$ of $\Delta_{n-1,r-1}$. In other words, for every edge $(X,Y)$ of
$\Delta_{n,r}' \subseteq \Delta_{n,r}$, $(X,Y)$ is an initial blue edge of $\Delta_{n,r}$ if and only if
$(\widehat{X},\widehat{Y})$ is an initial blue edge of $\Delta_{n-1,r-1}$. (Here, in each case, by an initial blue edge we mean an
edge that is blue by virtue of being in the spanning tree.)

Now consider an arbitrary edge $(X,Y)$ of $\Delta_{n,r}$ such that $(X,Y)$ belongs to $\Delta_{n,r}'$. Then from the
above observations $\Delta_{n,r}'$ is an isomorphic copy of $\Delta_{n-1,r-1}$, preserving the initial red and blue
edge colours, and since by induction we are assuming that Proposition~\ref{lem_bipartitecolours} holds for
$\Delta_{n-1,r-1}$, it follows immediately that using the same sequence of elementary edge transformations inside
$\Delta_{n,r}'$ we can transform $(X,Y)$ into a blue edge.

\

\begin{quote}
\begin{itemize}
\item[($\dag$)] \emph{Therefore by induction we may assume that every edge $(X,Y)$ of $\Delta_{n,r}$ in $\Delta_{n,r}'$
has already been turned blue. This assumption will remain in place for the rest of the section. }
\end{itemize}
\end{quote}

\

\begin{lem}
\label{lem_ones} Let $(X,Y)$ be an edge in $\Delta_{n,r}$. If $(X,Y)$ belongs to the subgraph
\begin{align}
\label{eqn_1start}
(1 < i_2 < \cdots < i_r)
\times
(1 < j_2 < \cdots < j_r),
\end{align}
then $(X,Y)$ can be turned blue.
\end{lem}
\begin{proof}

Let $(X,Y)$ be an edge in the subgraph \eqref{eqn_1start}.
Let $X'$ be the matrix obtained by replacing the first column of $X$ by the $n \times 1$ vector $ [1,0,0,\ldots,0]^T $
and let $Y'$ be the matrix obtained by replacing the first row of $Y$ by the $1 \times n$ vector $ [1,0,0,\ldots,0]. $
Note that $Y'$ is still a matrix in the set $\Y_r$ (that is, it is still a RRE rank $r$ matrix) and $Y'$ belongs to the
same region as $Y$. Similarly $X' \in \X_r$ and $X'$ belongs to the same region as $X$. Since $YX = I_r$, it follows
from the way that $Y'$ and $X'$ have been defined that $YX' = I_r$, $Y'X = I_r$ and $Y'X'=I_r$. Therefore the vertices
$\{ X,X',Y,Y' \}$ induce a square in $\Delta_{n,r}$. Now the edge $(X',Y')$ belongs to $\Delta_{n,r}'$ and hence may be
turned blue by induction ($\dagger$). Since $(X',Y')$ is blue, and $X$ and $X'$ belong to the same region, it follows
from Lemma~\ref{lem_useful}(i) that $(X,Y')$ may be turned blue. Dually, since $Y$ and $Y'$ belong to the same region,
$(X',Y)$ may be turned blue. Therefore the remaining edge $(X,Y)$ in the square may be turned blue, completing the
proof of the lemma. \end{proof}

\begin{lem}
\label{lem_Ireduction2} Every edge $(X,Y)$ in the subgraph
\begin{align}\label{eqn_firstsame}
({i_1} < i_2 < \cdots < i_r)
\times
({i_1} < j_2 < \cdots < j_r)
\end{align}
can be turned blue.
\end{lem}
\begin{proof}
Let $(X,Y)$ be an edge in the subgraph \eqref{eqn_firstsame}. If $i_1=1$ we are done by Lemma~\ref{lem_ones}, so
suppose $i_1
> 1$. Let $X'$ be the matrix obtained by replacing the first row of $X$ by the $1 \times r$ vector $ [1,0,0\ldots,0] $,
and let $Y'$ be the matrix obtained by replacing the first row of $Y$ by the $1 \times n$ vector $ [1,0,0\ldots,0] $.
Note that since the $i_1$th column of $Y$ is the $r \times 1$ vector $[1,0,0,\ldots,0]^T$, and since $i_1 > 1$, this
transformation means that the $i_1$th column of $Y'$ is the zero vector. Clearly $Y' \in \Y_r$ and $X' \in \X_r$.
Since $i_1>1$ it follows that $YX' = YX = I_r$. This in turn, along with the definition of $Y'$, implies $Y'X'=YX'=I_r$.
Therefore each of $(X,Y)$, $(X',Y)$ and $(X',Y')$ is an edge in $\Delta_{n,r}$ (while $(X,Y')$ is not an edge since
$Y'X \neq I_r$).

Next consider the subgraph of $\Delta_{n,r}$ induced by the four vertices
\begin{align*}
\Y_r: & & Y & & & Y' \\
\X_r: & & X' & & & I(1,i_1|j_2|j_3|\cdots|j_r)^T.
\end{align*}
Straightforward computations show that these four vertices form a square in $\Delta_{n,r}$. In this square, both of the
edges $(X',Y')$ and $(I(1,i_1|j_2|j_3|\cdots|j_r)^T, Y')$ may be turned blue by Lemma~\ref{lem_ones}, while the edge
$(I(1,i_1|j_2|j_3|\cdots|j_r)^T, Y)$ belongs to the subgraph
\[
(1 < {j_2 < j_3 < \cdots < j_r})
\times
(i_1 < {j_2 < j_3 < \cdots < j_r}),
\]
and so can be turned blue by Lemma~\ref{lem_closer}.
Therefore we deduce that the edge $(X',Y)$ may be turned blue.

Finally consider the subgraph of $\Delta_{n,r}$ induced by the four vertices
\begin{align*}
\Y_r: & & Y & & & I(i_1|i_2|i_3|\cdots|i_r) \\
\X_r: & & X & & & X'.
\end{align*}
Again, it is easily verified that this set of vertices induces a square in $\Delta_{n,r}$. We saw above that the edge $(X',Y)$ may be turned blue.
The edge
$
(X,I(i_1|i_2|i_3|\cdots|i_r))
$
belongs to a diagonal region and so may be turned blue by Lemma~\ref{lem_finishingoff}(i), while
the edge $ (X',I(i_1|i_2|i_3|\cdots|i_r)) $ belongs to the subgraph
\[
(1 < {i_2 < \cdots < i_r})
\times
(i_1 < {i_2 < \cdots < i_r}),
\]
and so may be turned blue by Lemma~\ref{lem_closer}. Since three of the four edges of the square can be turned blue, we deduce that the fourth edge $(X,Y)$ may be turned blue, completing the proof of the lemma.
\end{proof}
We are now in a position to complete the proof of the main result of this section.

\begin{proof}[Proof of Proposition~\ref{lem_bipartitecolours}]
Let $(X,Y)$ be an arbitrary edge of $\Delta_{n,r}$, where $(X,Y)$ belongs to
\[
(i_1 < i_2 < \cdots < i_r)
\times
(j_1 < j_2 < \cdots < j_r),
\]
say. If $i_1 = j_1$ we are done by Lemma~\ref{lem_Ireduction2}, so suppose $i_1 > j_1$ (the other case may be dealt
with using a dual argument). Let $X'$ be the matrix obtained by replacing the first column of $X$ by the $n \times 1$
vector $ [0,0,\ldots,0,1,0,\ldots,0]^T $ with $1$ in position $j_1$ and $0$s elsewhere. Note that since row $i_1$ of
$X$ is the $1 \times r$ vector $ [1,0,0,\ldots,0], $ it follows that row $i_1$ of $X'$ is the zero vector. Clearly
since $i_1 > j_1$ it follows that $X' \in \X_r$. Now consider the subgraph of $\Delta_{n,r}$ induced by the four
vertices
\begin{align*}
\Y_r: & & Y & & & I(j_1, i_1|i_2|i_3|\cdots|i_r) \\
\X_r: & & X & & & X'.
\end{align*}
From the definition of $X'$ it follows that $YX' = YX = I_r$, and it is then easily checked that
these four vertices induce a square in $\Delta_{n,r}$.
The edge \[(X,I(j_1, i_1|i_2|i_3|\cdots|i_r)) \] belongs to
\[
(i_1 < {i_2 < \cdots < i_r})
\times
(j_1 < {i_2 < \cdots < i_r})
\]
and so may be turned blue by Lemma~\ref{lem_closer}. The edge $ (X',I(j_1, i_1|i_2|i_3|\cdots|i_r)) $ belongs to
\[
(j_1 < i_2 < \cdots < i_r)
\times
(j_1 < i_2 < \cdots < i_r),
\]
a diagonal region, and so may be turned blue by Lemma~\ref{lem_finishingoff}(i). Finally, the edge $ (X',Y) $ belongs
to
\[
({j_1} < i_2 < \cdots < i_r)
\times
({j_1} < j_2 < \cdots < j_r)
\]
and so may be turned blue by Lemma~\ref{lem_Ireduction2}.
Since all three of these edges may be turned blue we deduce that the remaining edge $(X,Y)$ of this square may be turned blue, completing the proof of the proposition.
\end{proof}

\section{Combinatorial Properties of Multiplication Tables}
\label{sec_multtables}

From the explanation of Stage~2 of the proof of our main result given in
Section~\ref{sec_outline} it may be seen that establishing this part of the proof comes down to the combinatorial analysis of the structure matrix $P_r$ for the Rees matrix representation of $D_r^0$.
Even though in the statement of the main result Theorem~\ref{thm_main} we insist that $r < n/3$, all of the results in this section will be proved under the weaker assumption that $1 \leq r < n/2$ and this assumption about the relationship between $n$ and $r$ will remain in place throughout the section.

Recall that $P_r = (P_r(Y,X))$ is a matrix with rows indexed by $\Y_r$, columns by $\X_r$, and entries $P_r(Y,X) = YX$
if $YX$ is of rank $r$, and $0$ otherwise. So the entries of $P_r$ come from the set $GL_r(Q) \cup \{ 0 \}$. We also
view the abstract generators $\mathcal{F}$ given in \eqref{eqn_generators} as being arranged in a table also with rows
indexed by $\Y_r$, columns indexed by $\X_r$ and the entry $(Y,X)$ is $f_{X,Y}$ if $P_r(Y,X) \neq 0$ (i.e. if $f_{X,Y}
\in \mathcal{F}$) and $0$ otherwise. So far, using the defining relations \eqref{eqn_middle}--\eqref{eqn_bottom} from the
presentation $\mathcal{P}_{r,n}$ we have been making deductions about relations between the symbols $f_{X,Y}$ appearing
in this table. The results from the previous section show that we may deduce $f_{X,Y}=1$ whenever the corresponding
entry $P_r(Y,X)$ of $P_r$ satisfies $P_r(Y,X)=I_r$. That was Stage~1 of the proof. Now we move on to consider Stage~2
of the proof. In this stage our aim is to prove that for any pair of non-zero entries $P_r(Y,X)$ and $P_r(Y',X')$ from
the structure matrix $P_r$, if $P_r(Y,X) = P_r(Y',X')$ then $f_{X,Y}=f_{X',Y'}$ may be deduced from
\eqref{eqn_middle}--\eqref{eqn_bottom}.

As we did for the bipartite graph $\Delta_{n,r}$ in the previous section, we shall partition the matrix $P_r$ into
regions
\begin{align}
\label{eqn_region2}
(i_1 < i_2 < \cdots < i_r)
\times
(j_1 < j_2 < \cdots < j_r)
\end{align}
where the region \eqref{eqn_region2} is the set of all pairs $(Y,X)$ with $Y \in \Y_r$, $X \in \X_r$, $LC(Y) = \{ i_1,
i_2, \ldots, i_r \}$ and $LR(X) = \{ j_1, j_2, \ldots, j_r \}$. By the entries in the region \eqref{eqn_region2} we
mean the set of all matrices $YX$
with rank $r$
where $(Y,X)$ belongs to the region \eqref{eqn_region2}. In this section we focus our
attention just on the region
\begin{align}
\label{eqn_region}
(1 < 2 < \cdots < r)
\times
(1 < 2 < \cdots < r),
\end{align}
whose entries are those of the form:
\[
YX = [I_r|A]\left[\frac{I_r}{B}\right] = I_r + AB,
\]
where $A \in \mathrm{Mat}_{r \times (n-r-1)}(Q)$ and $B \in \mathrm{Mat}_{(n-r-1) \times r}(Q)$.
The aim of this section is to prove the following result.
\begin{lem}
\label{lem_topleft}
Let $n$ and $r$ be positive integers with $1 \leq r < n/2$.
Let $Y, Y' \in \Y_r$ and $X, X' \in \X_r$ with $LC(Y) = LC(Y') = LR(X) = LR(X') = \{1,2,\ldots, r\}$.
If $P_r(Y,X) = P_r(Y',X') \neq 0$ then $f_{X,Y}=f_{X',Y'}$ is a consequence of the relations \eqref{eqn_middle}--\eqref{eqn_bottom}.
\end{lem}
As in the previous section, we shall find it useful to recast this problem in purely combinatorial terms before solving it. We begin by introducing a general framework, and some terminology, for the analysis of combinatorial properties of tables.

Let $P = P(B,A)$ be a matrix with rows indexed by a set $\mathcal{B}$ and columns indexed by $\mathcal{A}$, where the
entries of $P$ all come from a set $L$, that we call the set of labels. Then given an element $l \in L$ we define a
graph, called the \emph{$\lambda$-graph of $l$}, with

\

\noindent \textbf{Vertices:} $\{ (B,A) \in \mathcal{B} \times \mathcal{A} : P(B,A) = l \}$: the set of all coordinates with label $l$, and

\

\noindent \textbf{Edges:} $(B,A)$ and $(B',A')$ are joined by an edge if and only if $B=B'$ or $A=A'$.

\

\noindent So, the $\lambda$-graph of $l \in L$ is obtained by removing all entries from the matrix except occurrences of the symbol $l$, and then drawing an edge between every pair of $l$s that belong to the same row, or to the same column.

Now, one natural source of such matrices is given by the multiplication tables of semigroups, where given a semigroup $S$ we take $\mathcal{A} = \mathcal{B} = L = S$ and define the entry $P(s,t) = st$. Let us briefly think about how $\lambda$-graphs behave in this situation. If $S$ happens to be a group, $S=G$, then this matrix is a Latin square and so (unless the group is trivial) for every $g \in L=G$ the $\lambda$-graph of $g$ will not be connected (in fact it will not have any edges at all).
On the other hand, if $S$ is a semigroup with a zero element $0 \in S$, then since in the multiplication table the row labelled by $0$ (and dually column labelled by $0$) contains all zeros, it is clear that in this case the $\lambda$-graph of $0$ in the multiplication table is connected.
Now suppose that $S$ is monoid with a non-trivial group of units such that the set of non-invertible elements of $S$ forms an ideal of $S$ (for example, the semigroup $M_n(Q)$ has this property). Then since here a product $st$ is invertible if and only if both $s$ and $t$ are, by the same reasoning as for groups above, the $\lambda$-graphs of the invertible elements $s \in S$ (i.e. those elements from the group of units of $S$) will not be connected. So for such a semigroup the most one could hope for would be for the $\lambda$-graphs of every non-invertible element to be connected. As we shall see below, this is exactly what happens in
the multiplication table of the semigroup $M_n(Q)$. In fact we show rather more than this.
\begin{thm}
\label{thm_multtables} Let $k$ and $m$ be positive integers with $k \leq m$, let $\mathcal{B}$ be the set of all $k
\times m$ matrices over a division ring $Q$, $\mathcal{A}$ be the set of $m \times k$ matrices over $Q$, and let
$\mathbb{T}_{m,k} = \mathbb{T}_{m,k}(B,A)$ be the matrix with entries $BA \in M_k(Q)$ where $B \in \mathcal{B}$ and $A
\in \mathcal{A}$. Let $K \in M_k(Q)$ be arbitrary.
\begin{enumerate}
\item[(i)] If $k <m$ then the $\lambda$-graph of $K$ in $\mathbb{T}_{m,k}$ is connected.
\item[(ii)] If $k=m$ and $K$ is non-invertible then the $\lambda$-graph of $K$ in $\mathbb{T}_{m,k}$ is connected.
\end{enumerate}
\end{thm}
It should be noted that, in contrast to the Rees structure matrix $P_r$, in the matrix $\mathbb{T}_{m,k}$ the index
sets $\mathcal{A}$ and $\mathcal{B}$ range over \emph{all possible matrices}, not just those in RRE form, and all products
$BA$ are recorded in the table, including those with rank less than $k$.

Theorem~\ref{thm_multtables} is a general result which is possibly of independent interest. It might be of interest to explore more which semigroups have multiplication tables with this property, and whether there is some general connection between semigroups with this property and those for which the maximal subgroups of $IG(E)$ are well behaved.

Before proving Theorem~\ref{thm_multtables} let us see how it can be used to obtain Lemma~\ref{lem_topleft} as a corollary.
Clearly
\[
\{ Y \in \Y_r : LC(Y) = \{1,2,\ldots, r\} \}
=
\{
\left[%
\begin{array}{cc}
  I_{r} & \bar{Y}
\end{array}%
\right]: \bar{Y} \in M_{r \times (n-r)}
\}
\]
and the natural map $Y \mapsto \bar{Y}$ where $Y = [I_r \; \bar{Y}]$ defines a bijection between the set of $Y \in \Y_r$ with
$LC(Y) = \{1,2,\ldots, r\}$ and the set $M_{r \times (n-r)}$ of all $r \times (n-r)$ matrices over $Q$.
Also, for $Y \in \Y_r$ and $X \in \X_r$, with $LC(Y) = LR(X) = \{1,2,\ldots, r\}$, writing
\[
Y =
\left[%
\begin{array}{cc}
  I_{r} & \bar{Y}
\end{array}%
\right],
\quad
X =
\left[%
\begin{array}{c}
  I_{r} \\ \bar{X}
\end{array}%
\right]
\]
we have
$
YX = I_r + \bar{Y}\bar{X}.
$
Thus for every pair $(Y,X), (Y',X') \in \Y_r \times \X_r$,
with $LC(Y) = LC(Y') = LR(X) = LR(X') = \{1,2,\ldots, r\}$,
we have
\begin{align}
\label{eqn_connection}
YX = Y'X'
\Leftrightarrow
\bar{Y}\bar{X} = \bar{Y'}\bar{X'}.
\end{align}

\begin{proof}[Proof of Lemma~\ref{lem_topleft}]
Let $(Y,X)$ be an arbitrary pair in the region
\begin{align}
\label{eqn_thecomponent} (1 < 2 < \cdots < r) \times (1 < 2 < \cdots < r)
\end{align}
such that $P_r(Y,X) \neq 0$, that is, $\rank{YX} = r$.
It follows from Theorem~\ref{thm_multtables}, with $k=r$ and $m=n-r > r = k$ (since by assumption $r < n/2$),
that the $\lambda$-graph of $\bar{Y}\bar{X}$ in $\mathbb{T}_{m,k} = \mathbb{T}_{n-r,r}$ is connected. But then from
\eqref{eqn_connection} it follows that the $\lambda$-graph of $YX$ in the region \eqref{eqn_thecomponent}
is connected. Since $(Y,X)$ were arbitrary we obtain that for every such entry $YX$ in the region
\eqref{eqn_thecomponent} the $\lambda$-graph of this matrix in the component \eqref{eqn_thecomponent} is connected.

Now let $Y, Y' \in \Y_r$ and $X, X' \in \X_r$ with $LC(Y) = LC(Y') = LR(X) = LR(X') = \{1,2,\ldots, r\}$ and $P_r(Y,X)
= P_r(Y',X') \neq 0$.

If $Y = Y'$ then
\[
\begin{array}{cl||cc}
    &           &       X
    &           X'

    \\ \hline\hline
    & & &  \\
& Y &       YX      &   YX' \\
 & & &   \\
& I(1|2| \cdots |r)
&       I_r       &   I_r                                         \\
\end{array}
\]
is a singular square by Theorem \ref{rectsing}, equation \eqref{eq_99}, and the fact that $YX = YX'$, and hence from
relation \eqref{eqn_bottom} we deduce $f_{X,Y} = f_{X',Y}$ in this case. Dually, if $X = X'$ then the square
\[
\begin{array}{cl||cc}
    &           &       X
    &           I(1|2| \cdots |r)^T

    \\ \hline\hline
    & & &  \\
& Y
&       YX      &   I_r                                         \\
 & & &   \\
 & Y'
&       Y'X      &   I_r  \\
\end{array}
\]
is singular since $YX = Y'X$, and hence from relation \eqref{eqn_bottom} we deduce $f_{X,Y} = f_{X,Y'}$ in this case.

But now, since we know that the $\lambda$-graph of $YX$ in \eqref{eqn_thecomponent}
is connected, it follows that there is a sequence of entries in \eqref{eqn_thecomponent} from $P_r(Y,X)$ to
$P_r(Y',X')$, all equal to $YX$, where adjacent terms in the sequence are either in the same row or column of $P_r$,
and thus the corresponding generators are equal by the arguments given in the previous two paragraphs. Therefore, we
may deduce $f_{X,Y} = f_{X',Y'}$ as a consequence of the relations from the presentation $\mathcal{P}_{r,n}$.
\end{proof}
The rest of this section is concerned with the proof of the above theorem.

\begin{proof}[Proof of Theorem~\ref{thm_multtables}]
Let $k$ and $m$ be positive integers with $k \leq m$. We prove the result by induction on $k+m$. When $k=m=1$ the
result is trivially seen to hold, since in this case $\mathbb{T}_{1,1}$ is the multiplication table of $Q$, the only
non-invertible element of which is $0$, and as already observed above the corresponding $\lambda$-graph is connected.
Now suppose $k+m>2$ and assume inductively that the result holds for all pairs $(k',m')$ with $k' \leq m'$ and $k'+m' <
k+m$. Let $K \in M_k(Q)$ be arbitrary.

The table $\mathbb{T}_{m,k}$ naturally divides into regions indexed by pairs $(\alpha, \beta)$ where by definition the
$(\alpha,\beta)$-region is the set of all pairs
\[
\left( [\alpha|A'],\left[\frac{\beta}{B'}\right] \right),
\]
where $A' \in M_{k \times (m-1)}(\F)$, $B' \in M_{(m-1) \times k}(\F)$, $\alpha$ is a column vector and $\beta$ is a row
vector. Note that the region $(0,0)$ is a natural copy of the table $\mathbb{T}_{m-1,k}$ inside $\mathbb{T}_{m,k}$.

For part (i), we are given that $k<m$ and must prove that the $\lambda$-graph of $K$ in $\mathbb{T}_{m,k}$ is
connected. We consider two cases.

\

\noindent \emph{Case 1: $k<m-1$:} Let $A \in  \mathrm{Mat}_{k \times m}(Q)$ and $B \in  \mathrm{Mat}_{m \times k}(Q)$ be arbitrary, and write
\[
A = [\alpha | A_1 | A_2] \quad \mbox{and} \quad B = \left[ \begin{array}{c} \beta \\ \hline B_1 \\ \hline B_2
\end{array} \right],
\]
where $\alpha$ is a $k \times 1$ column vector, $\beta$ is a $1 \times k$ row vector, and $A_2, B_2 \in M_k(\F)$. Then
\begin{equation}\label{eqn_theproduct}
AB = \alpha \beta + A_1 B_1 + A_2 B_2.
\end{equation}
We begin by arguing that without loss of generality we may assume that $B_2 \in M_k(\F)$ is invertible. Indeed, let $U
\in M_k(Q)$ be an idempotent $\gr$-related to $A_2 B_2$.
Such an idempotent $U$ exists since $M_k(Q)$ is regular.
Then, since every idempotent is a left identity in its
$\gr$-class (see \cite[Proposition~2.3.3]{howie95}), $U \gr A_2 B_2$ implies $U A_2 B_2 = A_2 B_2$ and hence also $U
A_2 \gr U \gr A_2 B_2$. Therefore by \eqref{eqn_R} there is an invertible matrix $X \in GL_k(Q)$ such that
$ U A_2 X = A_2 B_2 = U A_2 B_2. $ Thus
\begin{align*}
[\alpha | A_1 | A_2] \left[ \begin{array}{c} \beta \\ \hline B_1 \\ \hline B_2 \end{array} \right] & = [\alpha | A_1 |
UA_2] \left[ \begin{array}{c} \beta \\ \hline B_1 \\ \hline B_2 \end{array} \right]
\\
& = [\alpha | A_1 | UA_2] \left[ \begin{array}{c} \beta \\ \hline B_1 \\ \hline X \end{array} \right],
\end{align*}
where $X \in M_k(\F)$ is invertible, and this sequence of equalities defines a path in the $\lambda$-graph of $AB$.
Hence we may assume without loss of generality that $B_2$ is invertible. But then
\begin{align*}
[\alpha | A_1 | A_2] \left[ \begin{array}{c} \beta \\ \hline B_1 \\ \hline B_2 \end{array} \right] & = [0 | A_1 | A_2 +
\alpha\beta B_2^{-1}] \left[ \begin{array}{c} \beta \\ \hline B_1 \\ \hline B_2 \end{array} \right]
\\
& = [0 | A_1 | A_2 + \alpha\beta B_2^{-1}] \left[ \begin{array}{c} 0 \\ \hline B_1 \\ \hline B_2 \end{array} \right],
\end{align*}
and so we have found a $\lambda$-path into the $(0,0)$-region. Recall that the $(0,0)$-region is a natural copy of
$\mathbb{T}_{m-1,k}$ inside the table $\mathbb{T}_{m,k}$.
Since $k < m-1$ it follows by induction, applying (i), that the $\lambda$-graph of $AB$ restricted to the $(0,0)$-region is connected. Therefore every occurrence of $AB$ is connected to an occurrence of $AB$ in the $(0,0)$-region, while any two occurrences of $AB$ in the $(0,0)$-region are joined by a $\lambda$-path in the $(0,0)$-region by induction. Since the pair $A$, $B$ was arbitrary, this completes the proof that the $\lambda$-graph of $K$ is connected in this case.

\

\noindent \emph{Case 2: $k=m-1$:}
Arguing as in the previous case, for every entry in $\mathbb{T}_{m,k}$ there is a $\lambda$-path to a pair of the form
\[
[0|C]\left[\frac{0}{D}\right],
\]
where $C, D \in \mathrm{Mat}_{(m-1) \times (m-1)}(Q)$ and $D$ is invertible. Now there are two cases depending on whether or not $C$ is invertible.

If $C$ is not invertible then $CD$ is not invertible and so by induction, applying (ii),
the $\lambda$-graph of $CD$ in the $(0,0)$-region is connected, and the proof is complete as in the previous case.

So we may suppose that both $C$ and $D$ are invertible, and hence so is their product $CD$. It is easy to see that for
any matrix $L \in M_{m-1}(Q)$ appearing in the table $\mathbb{T}_{m,k} = \mathbb{T}_{m,m-1}$ and for any pair $X$ and
$Y$ of invertible $(m-1) \times (m-1)$ matrices we have that the $\lambda$-graph of $L$ is connected if and only if the
$\lambda$-graph of $XL$ is connected if and only if the $\lambda$-graph of $XLY$ is connected. Indeed, left
multiplication by $X$ induces a permutation of the set of matrices $M_{(m-1) \times m}(Q)$ which label the rows of the
table $\mathbb{T}_{m,m-1}$; the same is true for right multiplication by $Y$ on the set of matrices $M_{m \times
(m-1)}(Q)$ labelling the columns of $\mathbb{T}_{m,m-1}$. This transformation of the table will result in a table where
the entries $XLY$ appear in precisely the positions where the entries $L$ appeared in the the original table
$\mathbb{T}_{m,m-1}$. Since permuting rows and columns of the table does not affect $\lambda$-connectedness, we have
the desired conclusion.

Therefore it will suffice to show that the $\lambda$-graph of
\[
[0|C]\left[\frac{0}{C^{-1}}\right] = I_{m-1},
\]
is connected. Of course within the $(0,0)$-region the $\lambda$-graph of $I_{m-1}$ is \emph{not} connected
(since $I_{m-1}$ belongs to the group of units) and so it will be necessary to move out of that region in order to
prove that the $\lambda$-graph of $I_{m-1}$ is connected in $\mathbb{T}_{m,m-1}$.

We shall prove that there is a $\lambda$-path connecting $([0|C]\left[\frac{0}{C^{-1}}\right])$ into the region
\[([1,0,0,\ldots,0], [1,0,0,\ldots,0]^T).\] Indeed, we have
\begin{align*}
[0|C]\left[\frac{0}{C^{-1}}\right] & = [0|C] \left[\frac{\begin{matrix} 1 & 0 & 0 & \cdots & 0
\end{matrix}}{C^{-1}}\right]
\\
\\
& = \left[ \begin{array}{cc}
  \begin{matrix}
  1 \\ 0 \\ 0 \\ \vdots \\ 0
  \end{matrix} \ \vline
  &
  (I_{m-1} - E_{11}) C
\end{array}\right] \left[\frac{\begin{matrix} 1 & 0 & 0 & \cdots & 0 \end{matrix}}{C^{-1}}\right],
\end{align*}
where $E_{11} = [1,0,\ldots,0]^T[1,0,\ldots,0]$ denotes the $(m-1) \times (m-1)$ matrix with a $1$ in the top left corner and zeros
everywhere else. Computing the last of these products gives
\[
E_{11} + (I_{m-1}-E_{11})CC^{-1} = I_{m-1},
\]
as required. But the matrix $(I_{m-1}-E_{11}) \in M_{m-1}(Q)$ is clearly not invertible and so it follows by induction,
applying (ii), that inside the region \[([1,0,0,\ldots,0], [1,0,0,\ldots,0]^T)\] the $\lambda$-graph of the $I_{m-1}$ is
connected. This is because the \[([1,0,0,\ldots,0], [1,0,0,\ldots,0]^T)\] region is a copy of the table
$\mathbb{T}_{m-1,m-1}$ with $E_{11}$ added to each entry, and therefore the $\lambda$-graph of $ I_{m-1} = E_{11} +
(I_{m-1}-E_{11}) $ is connected, since $I_{m-1}-E_{11}$ is non-invertible.

\begin{sloppypar}
In conclusion we have proved that for every occurrence of $I_{m-1}$ in $\mathbb{T}_{m,k}$ there is a $\lambda$-path
into the $(0,0)$-region, and for every occurrence of $I_{m-1}$ in the $(0,0)$-region there is a $\lambda$-path to the
$([1,0,0,\ldots,0], [1,0,0,\ldots,0]^T)$-region, and in this region every pair of occurrences of $I_{m-1}$ are
connected by a $\lambda$-path. Therefore the $\lambda$-graph of $I_{m-1}$ in $\mathbb{T}_{m,k}$ is connected,
completing the proof of the inductive step for part (i) of the theorem.
\end{sloppypar}

For part (ii), we are given that $k=m$ and that $K$ is non-invertible, and  again we want to show that the
$\lambda$-graph of $K$ in $\mathbb{T}_{m,k}=\mathbb{T}_{m,m}$ is connected.

Consider the entry $AB$ in the multiplication table where $A, B \in M_m(\F)$ and $AB$ is not invertible, so $\rank{AB}
= l < m=k$. Therefore $AB$ is in the same $\gd$-class as the matrix $ J =
\begin{bmatrix}
I_l & 0 \\
0 & 0
\end{bmatrix}.
$ Hence by \eqref{eqn_J} we can write $J = X (AB) Y$ where $X$ and $Y$ are invertible matrices. But since $X$ and $Y$
are invertible it follows that in $\mathbb{T}_{m,k}$ the $\lambda$-graph of $AB$ is connected if and only if the $\lambda$-graph
of $XAB$ is connected if and only if the $\lambda$-graph of $XABY=J$ is connected.
So we shall prove instead that the $\lambda$-graph of $J$ is connected.

Suppose that $AB = J$ where $A, B \in M_m(Q)$. Then we can write
\[
AB =
\begin{bmatrix}
A_{11} & A_{12} \\
A_{21} & A_{22}
\end{bmatrix}
\begin{bmatrix}
B_{11} & B_{12} \\
B_{21} & B_{22}
\end{bmatrix}
=
\begin{bmatrix}
I_l & 0 \\
0 & 0
\end{bmatrix}
=
J
,
\]
where $A_{11}$ and $B_{11}$ are both $l \times l$ matrices. Consequently, there is a $\lambda$-path given by
\begin{align*}
\begin{bmatrix}
A_{11} & A_{12} \\
A_{21} & A_{22}
\end{bmatrix}
\begin{bmatrix}
B_{11} & B_{12} \\
B_{21} & B_{22}
\end{bmatrix}
& =
\begin{bmatrix}
A_{11} & A_{12} \\
A_{21} & A_{22}
\end{bmatrix}
\begin{bmatrix}
B_{11} & 0 \\
B_{21} & 0
\end{bmatrix} \\
& =
\begin{bmatrix}
A_{11} & A_{12} \\
0 & 0
\end{bmatrix}
\begin{bmatrix}
B_{11} & 0  \\
B_{21} & 0
\end{bmatrix},
\end{align*}
into a region that is a natural copy of $\mathbb{T}_{m,l}$  inside $\mathbb{T}_{m,k}$. By induction, since $l < k = m$,
the $\lambda$-graph of $AB$ in this copy of $\mathbb{T}_{m,l}$  in $\mathbb{T}_{m,k}$ is connected, which completes the
proof of the inductive step for (ii), and hence also completes the proof of the theorem.
\end{proof}

\section{Strongly connecting the Table}
\label{sec_lambdalabels}

In this section we shall complete Stage~2 of the proof of the main theorem by extending Lemma~\ref{lem_topleft} to
obtain the following result.
Throughout this section $n$ and $r$ will denote positive integers satisfying $1 \leq r < n/3$. This assumption will be necessary for our proof of Theorem~\ref{thm_thefulltable} below.

\begin{lem}
\label{lem_fulltable}
Let $n$ and $r$ be positive integers with $1 \leq r < n/3$, and let $Y, Y' \in \Y_r$ and $X, X' \in \X_r$. If $P_r(Y,X) = P_r(Y',X') \neq 0$ then
$f_{X,Y}=f_{X',Y'}$ is a consequence of the relations \eqref{eqn_middle}--\eqref{eqn_bottom}.
\end{lem}

As usual, we first recast this problem combinatorially.

Let $P = P(B,A)$ be a matrix with rows indexed by a set $\mathcal{B}$ and columns indexed by $\mathcal{A}$, where the
entries of $P$ all come from a set $L \cup \{ 1 \}$ where $1$ is a distinguished symbol not belonging to $L$. Let $l
\in L$ and consider the $\lambda$-graph of $l$ defined in Section~\ref{sec_multtables}. We say that two vertices
$(B,A)$ and $(B',A')$ of the $\lambda$-graph of $l$ are \emph{connected by a strong edge} if either
\begin{enumerate}
\item[(i)]
$B=B'$ and there exists $B_1 \in \mathcal{B}$ such that $P(B_1,A) = P(B_1,A')=1$; or
\item[(ii)]
$A=A'$ and there exists $A_1 \in \mathcal{A}$ such that $P(B,A_1)=P(B',A_1)=1$.
\end{enumerate}
A \emph{strong path} is then a sequence of vertices where adjacent terms in the sequence are connected by strong edges,
and we say that the $\lambda$-graph of $l$ is \emph{strongly connected} if between any pair of vertices $(B,A)$ and
$(B',A')$ there is a strong path.
A strong path of length $3$ is illustrated in Figure~\ref{fig_table}.


\begin{figure}[t]

\begin{center}
\begin{tikzpicture}[scale=0.3,
smallbox/.style={draw, color=gray!40, fill=gray!35, rectangle, minimum height=30mm, minimum width=30mm},
bigbox/.style={draw, color=gray!20, fill=gray!20, rectangle, minimum height=60mm, minimum width=60mm}]
\tikzstyle{ACircle}=[circle,
                                    minimum size=0.25cm,
                                    draw=black!100,
                                    fill=gray!10]
\tikzstyle{ASquare}=[diamond,
                                    minimum size=0.25cm,
                                    draw=black!100,
                                    fill=gray!10]
\tikzstyle{ADiamond}=[rectangle,
                                    minimum size=0.25cm,
                                    draw=black!100,
                                    fill=gray!10]
\tikzstyle{vertex}=[circle,draw=black, fill=black, inner sep = 0.3mm]
\node at (10,20) [bigbox] {};
\node at (5,25) [smallbox] {};
\draw (0,30)--(30,30)--(30,0)--(0,0)--(0,30);
\draw (0,20)--(30,20);
\draw (0,14)--(30,14);
\draw (0,10)--(30,10);
\draw (0,2)--(30,2);
\draw (20,0)--(20,30);
\draw (16,0)--(16,30);
\draw (10,0)--(10,30);
\draw (28,0)--(28,30);
\node (I00) at (0.5,29.5) {\small $I$};
\node (I01) at (1.5,29.5) {\small $I$};
\node (I02) at (2.5,29.5) {\small $I$};
\node (dots1) at (5,29.5) {\small $\ldots$};
\node (I03) at (7.5,29.5) {\small $I$};
\node (I04) at (8.5,29.5) {\small $I$};
\node (I05) at (9.5,29.5) {\small $I$};
\node (I10) at    (0.5,28.5) {\small $I$};
\node (I20) at    (0.5,27.5) {\small $I$};
\node (vdots1) at (0.5,25) {\small $\vdots$};
\node (I30) at    (0.5,22.5) {\small $I$};
\node (I40) at    (0.5,21.5) {\small $I$};
\node (I50) at    (0.5,20.5) {\small $I$};
\node [ACircle] (K1c) at    (5.5,23.5) {};
\node (K1) at    (5.5,23.5) {\small $K$};
\node [ACircle] (K3c) at    (5.5,17.5) {};
\node (K3) at    (5.5,17.5) {\small $K$};
\node [ACircle] (II4c) at    (17.5,23.5) {};
\node (II4) at    (17.5,23.5) {\small $I$};
\node [ACircle] (II5c) at    (17.5,17.5) {};
\node (II5) at    (17.5,17.5) {\small $I$};
\node [ASquare] (K4s) at    (1.5,17.5) {};
\node (K4) at    (1.5,17.5) {\small $K$};
\node [ASquare] (K5s) at    (1.5,11.5) {};
\node (K5) at    (1.5,11.5) {\small $K$};
\node [ASquare] (II6s) at    (13.5,17.5) {};
\node (II6) at    (13.5,17.5) {\small $I$};
\node [ASquare] (II7s) at    (13.5,11.5) {};
\node (II7) at    (13.5,11.5) {\small $I$};
\node [ADiamond] (K2t) at    (11.5,4.5) {};
\node (K2) at    (11.5,4.5) {\small $K$};
\node [ADiamond] (K7t) at    (14.5,4.5) {};
\node (K7) at    (14.5,4.5) {\small $K$};
\node [ADiamond] (I8t) at    (14.5,19.5) {};
\node (I8) at    (14.5,19.5) {\small $I$};
\node [ADiamond] (I9t) at    (11.5,19.5) {};
\node (I9) at    (11.5,19.5) {\small $I$};
\node (I200) at    (10.5,19.5) {\small $I$};
\node (I201) at    (11.5,19.5) {\small $I$};
\node (I202) at    (12.5,19.5) {\small $I$};
\node (I203) at    (15.5,19.5) {\small $I$};
\node (dots3) at    (13.4,19.5) {\small $\ldots$};
\node (I210) at    (10.5,18.5) {\small $I$};
\node (I220) at    (10.5,17.5) {\small $I$};
\node (I240) at    (10.5,14.5) {\small $I$};
\node (vots3) at    (10.5,16.5) {\small $\vdots$};
\node (I300) at    (16.5,13.5) {\small $I$};
\node (I310) at    (19.5,13.5) {\small $I$};
\node (I301) at    (16.5,10.5) {\small $I$};
\node (dots4) at    (18,13.5) {\small $\ldots$};
\node (vots4) at    (16.5,12.5) {\small $\vdots$};
\node (I400) at    (29,1) {\small $I$};
\draw [thick] (K1c)--(K3c)--(K4s)--(K5s);
\draw [thick] (K2t)--(K7t);
\node (dotsdd) at    (24,6) {\small $\ddots$};
\node (dotsrr) at    (24,25) {\small $\cdots$};
\node (dotsud) at    (5,6) {\small $\vdots$};
\node (12TOr) at    (5,31) {\small $(1 < 2 < \cdots< r)$};
\node (topright) at    (18,31) {\small $(i_1 < i_2 < \cdots < i_r)$};
\node [rotate=90] (12TOrbottom) at    (-1,25) {\small $(1 < 2 < \cdots< r)$};
\node [rotate=90] (bottomleft) at    (-1,12) {\small $(i_1 < i_2 < \cdots < i_r)$};
\draw [decorate,decoration={brace,amplitude=10pt},xshift=0pt,yshift=-4pt,thick]
(0,32) -- (20,32) node [black,midway,yshift=-0.6cm]
{};
\node (BigTopLabel) at    (10,34) {\small $\bigcup_{j_r \leq n-r, \; k_r \leq n-r}
(j_1 < j_2 < \cdots < j_r)
\times
(k_1 < k_2 < \cdots < k_r)
$};
\end{tikzpicture}
\end{center}
\caption{An illustration of the table $\mathbb{T}_{n,r}$ from Theorem \ref{thm_thefulltable}.
The regions of the table are indicated, in each diagonal region the $I$s corresponding to the edges of type (T1) and (T2) from the spanning tree $T_{n,r}$ are indicated.
The diagonal regions vary in size with the bottom right diagonal region $\Delta(n-r+1 < \cdots < n)$ having just a single entry.
The singular square indicated by the quadruple of shaded squares illustrates the proof of Lemma~\ref{lem_zero}.
}\label{fig_table}
\end{figure}
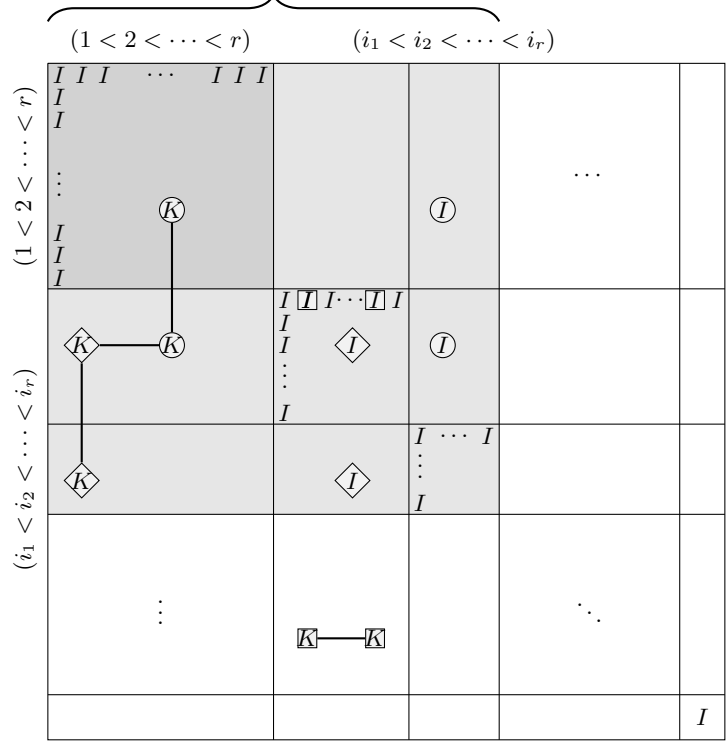

\begin{thm}
\label{thm_thefulltable} Let $r$ and $n$ be positive integers with $r < n/3$, let $\Y_r$ be the set of all $r \times n$
rank $r$ matrices over a division ring $Q$ in reduced row echelon form, $\X_r$ be the set of transposes of elements
of $\Y_r$, and let $\mathbb{T}_{n,r} = \mathbb{T}_{n,r}(Y,X)$ be the matrix with entries $YX \in M_r(Q)$ where $Y \in
\Y_r$ and $X \in \X_r$. Then for every matrix $K \in M_r(Q)$ the $\lambda$-graph of $K$ in $\mathbb{T}_{n,r}$ is
strongly connected with respect to the distinguished entries $1=I_r$.
\end{thm}
Note that $\mathbb{T}_{n,r}$ is not exactly the same as the Rees structure matrix $P_{r}$ since $\mathbb{T}_{n,r}$
contains all products $YX$ even if $YX$ does not have rank $r$.

The aim of  Theorem \ref{thm_thefulltable} is to show that for every symbol $K$ appearing in the table, the $\lambda$-graph of $K$ is strongly connected.
The structure of the proof is outlined in Figure~\ref{fig_table}.
A strong path of length $3$ is indicated in the figure. The first and last edges of this path are strong edges because of the singular squares indicated by the quadruple of diamonds and circles respectively. The remaining third edge of this path is a strong edge as a consequence of Lemma~\ref{lem_zero}. In Lemma~\ref{lem_topleft} we prove that the $\lambda$-graph of $K$ restricted to the small dark grey region of the table is strongly connected. Then in Corollary~\ref{cor_thebigbox} we prove that the  $\lambda$-graph of $K$ restricted to the larger light grey region of the table is strongly connected. This is done by finding a strong path from every $K$ in the light grey region to a label $K$ in the dark grey region. Finally we complete the proof of Theorem~\ref{thm_thefulltable} by finding a strong path from an arbitrary $K$ into the light grey region.

Before going on to prove Theorem~\ref{thm_thefulltable}
let us see how Lemma~\ref{lem_fulltable} may be deduced from it.

\begin{proof}[Proof of Lemma~\ref{lem_fulltable}]
Let $Y, Y' \in \Y_r$ and $X, X' \in \X_r$ and suppose that $P_r(Y,X) = P_r(Y',X') \neq 0$. If $(Y,X)$ and $(Y',X')$ are
connected by a strong edge in the Rees structure matrix $P_{r}$ then applying Lemma \ref{lem_hash}, equation
\eqref{eq_99} and relation \eqref{eqn_bottom} we may deduce that $f_{X,Y}=f_{X',Y'}$. It follows that if $(Y,X)$ and
$(Y',X')$ are connected by a strong path in $P_{r}$ then we may deduce that $f_{X,Y}=f_{X',Y'}$. But by
Theorem~\ref{thm_thefulltable}, $(Y,X)$ and $(Y',X')$ are connected by a strong path in $\mathbb{T}_{n,r}$ and
therefore it is immediate from the definitions of $\mathbb{T}_{n,r}$ and $P_{r}$ that the same path is also a strong
path in $P_{r}$ connecting $(Y,X)$ and $(Y',X')$, proving the lemma.
\end{proof}
The rest of this section will, therefore, be devoted to the proof of Theorem~\ref{thm_thefulltable}. As usual, we
partition the table $\mathbb{T}_{n,r}$ into regions, where the region
\begin{align}
\label{eqn_arbregion}
(i_1 < \cdots < i_r) \times
(j_1 < \cdots < j_r)
\end{align}
is the set of all pairs $(Y,X)$ where $LC(Y) = \{i_1, \ldots, i_r\}$ and $LR(X) = \{j_1, \ldots, j_r \}$.

\begin{lem}
\label{lem_zero} If $(Y,X)$ and $(Y',X')$ belong to the same region of $\mathbb{T}_{n,r}$, with $YX=Y'X'$, and either
$Y=Y'$, or $X=X'$, then they are strongly connected in $\mathbb{T}_{n,r}$.
\end{lem}
\begin{proof}
Suppose that $Y=Y'$ and that $X$ and $X'$ both
belong to the region
$
(i_1 < i_2 < \cdots < i_r).
$
Then the square
\[
\begin{array}{cl||cc}
    &           &       X
    &           X'

    \\ \hline\hline
    & & &  \\
& I(i_1|i_2| \cdots |i_r)
&       I_r       &   I_r                                         \\
 & & &   \\
 & Y
&       YX      &   YX' \\
\end{array}
\]
in $\mathbb{T}_{n,r}$ shows that $(Y,X)$ and $(Y,X')$ are strongly connected. The other case is dual using the column
labelled by $I(i_1|i_2|\cdots|i_r)^T$.
\end{proof}
The following result extends Lemma~\ref{lem_topleft}.

\begin{lem}
\label{lem_frombefore}
If $(Y,X)$ and $(Y',X')$ both belong to the region
\[
(1 < 2 < \cdots < r) \times
(1 < 2 < \cdots < r),
\]
and $YX = Y'X'$,
then there is a strong path from $(Y,X)$ to $(Y',X')$ in $\mathbb{T}_{n,r}$.
\end{lem}
\begin{proof}
From the results in Section~\ref{sec_multtables} there is a path from $(Y,X)$ to $(Y',X')$ in
\[
(1 < 2 < \cdots < r) \times
(1 < 2 < \cdots < r),
\]
and then by Lemma~\ref{lem_zero} this path is actually a strong path.
\end{proof}
Therefore, to prove Theorem~\ref{thm_thefulltable} it will be sufficient to show that for every entry $(Y,X)$ in
$\mathbb{T}_{n,r}$ there is a strong path into the region
\[
(1 < 2 < \cdots < r) \times
(1 < 2 < \cdots < r),
\]
and this is what the rest of the proof will be focused on establishing.
\begin{lem}
\label{lem_lemmaAnew}
Let $K \in M_r(Q)$ be an entry in a region
\begin{align}
\label{eqn_theregagain}
(i_1 < i_2 < \cdots < i_r)
\times
(j_1 < j_2 < \cdots < j_r)
\end{align}
where $i_r \leq n-r$ and $j_r \leq n-r$.
Then there is a strong path in $\mathbb{T}_{n,r}$ from this entry to an entry
\[
K=
\begin{bmatrix}
I_{r \times (n-r)}(i_1|\cdots|i_r) \; | \;  D
\end{bmatrix}
\left[
\begin{array}{c}
\ \\
I_{r \times (n-r)}(j_1|\cdots|j_r)^T \\
\ \\
\hline \\
N \\
\
\end{array} \right]
\]
where $N \in M_r(Q)$ and $D \in GL_r(Q)$.
Dually there is a strong path to an entry of the same form but where
$D \in M_r(Q)$ and $N \in GL_r(Q)$.
\end{lem}
\begin{proof}
We prove the first statement, the second is proved using a dual argument. The proof has two steps which are illustrated
in Figures~\ref{fig_step1} and \ref{fig_step2} respectively. Let $[P|A], [Z|B]^T \in \mathrm{Mat}_{r \times n}(Q)$ be
arbitrary such that $A,B \in M_r(Q)$,
$([P|A], [Z|B]^T)$ belongs to the region \eqref{eqn_theregagain},
and
\[
[P|A] \left[ \begin{array}{c} Z \\\hline B \end{array} \right] = PZ + AB = K.
\]
We proceed along similar lines as in the proof of Theorem~\ref{thm_multtables}. We shall construct a path where the
entire path belongs to the region \eqref{eqn_theregagain},
and consequently by Lemma~\ref{lem_zero} this path will automatically be a strong path.

For the first step of the proof, $A, B \in M_{r}({Q})$ and since this semigroup is regular there is an idempotent $U$ with $U \gr AB$, so $UA \gr AB$ and then by \eqref{eqn_R} there is an invertible matrix $C \in GL_r(Q)$ satisfying
$
UAC = UAB = AB.
$
Since $C$ is invertible, $M_r(Q) C = M_r(Q)$ and so there exists a matrix $L \in M_r(Q)$ such that the equation
\[
I_{r \times (n-r)}(i_1|\cdots|i_r)
Z
+
LC
= K
\]
is satisfied. Combining these observations, in Figure~\ref{fig_step1} we construct a strong path in the region
\eqref{eqn_theregagain} from $ \left( [P|A], \left[ \begin{smallmatrix} Z \\\hline B \end{smallmatrix} \right] \right)
$ to $ \left( [I_{r \times (n-r)}(i_1|\cdots|i_r)|L], \left[ \begin{smallmatrix} Z \\\hline C \end{smallmatrix} \right]
\right). $

For the second step of the proof, we use a dual argument to find a strong path, in the same region, from
\[
\left( [I_{r \times (n-r)}(i_1|\cdots|i_r)|L], \left[ \begin{array}{c} Z \\\hline C \end{array} \right] \right)
\]
to
\[
\left( [I_{r \times (n-r)}(i_1|\cdots|i_r)|D], \left[ \begin{array}{c} I_{r \times (n-r)}(i_1|\cdots|i_r)^T \\\hline N
\end{array} \right] \right),
\]
where $D \in GL_r(Q)$. This path is given in Figure~\ref{fig_step2} where $V$ is an idempotent with $V \gl LC$, $D \in
GL_r(Q)$ and $ LC = LCV = DCV $. Here $D$ exists by \eqref{eqn_L} since $CV \gl LC$. Then, using the fact that $D$ is
invertible, $N \in M_r(Q)$ is chosen so that the equation
\[
I_{r \times (n-r)}(i_1|\cdots|i_r)
I_{r \times (n-r)}(i_1|\cdots|i_r)^T
+
DN
=
K
\]
is satisfied. This completes the proof of the lemma. \end{proof}
\renewcommand{\arraystretch}{0.5}
\begin{figure}[t]
{\small
\[
\begin{array}{cl||cc}
    &           &       \begin{array}{c}
    \left[ \begin{array}{c} \ \\ Z \\ \ \\ \hline \ \\ B  \end{array} \right]  \\
    \ \end{array}
    &           \begin{array}{c}
    \left[ \begin{array}{c} \ \\ Z \\ \ \\ \hline \ \\ C  \end{array} \right]  \\
    \ \end{array} \\
    \hline\hline
    & & &  \\
& \left[ \begin{array}{c} \ \ P \ \ |A \end{array} \right]
&       PZ + AB = K &                                         \\
& & &   \\
& \left[ \begin{array}{c} \ \ P \ \ |UA \end{array} \right]
&       PZ + UAB = K &  PZ + UAC = K   \\
& & &   \\
& \left[ \begin{array}{c} \; I_{r \times (n-r)}(i_1|\cdots|i_r) \; |L \end{array} \right]
&       &       K    \\
\end{array}
\]
}
\caption{Proof of Lemma~\ref{lem_lemmaAnew}: a strong path in the region
$(i_1 < \cdots < i_r)
\times
(j_1 < \cdots < j_r)$.
}
\label{fig_step1}
\end{figure}
\renewcommand{\arraystretch}{1}
\renewcommand{\arraystretch}{0.2}
\begin{figure}[t]
{\small
\[
\begin{array}{cl||ccc}
    &           &      \begin{array}{c} \left[ \begin{array}{c}  \ \\ Z \\ \ \\ \hline \ \\ C \end{array} \right] \\ \ \end{array}
    &         \begin{array}{c} \left[ \begin{array}{c} \ \\ Z \\ \ \\ \hline \ \\ CV \end{array} \right] \\ \ \end{array}
    &         \begin{array}{c}  \left[ \begin{array}{c} \ \\ I_{r \times (n-r)}(j_1|\cdots|j_r)^T \\ \ \\ \hline \ \\ N \end{array} \right] \\ \ \end{array}
    \\ \hline\hline
    & & & & \\
& \left[ \begin{array}{c} \; I_{r \times (n-r)}(i_1|\cdots|i_r) \; |L \end{array} \right]
&       K & K &                               \\
& & &   & \\
& \left[ \begin{array}{c} \; I_{r \times (n-r)}(i_1|\cdots|i_r) \; |D \end{array} \right]
&               &   K & K \\
\end{array}
\]
}
\caption{Proof of Lemma~\ref{lem_lemmaAnew}: a strong path in the region
$(i_1 < \cdots < i_r) \times (j_1 < \cdots < j_r)$.
}
\label{fig_step2}
\end{figure}
\renewcommand{\arraystretch}{1}
\begin{lem}
\label{lem_gettinginto1212}
Let $K$ be an entry of the form
\begin{align}
\label{eqn_newK} K = \left[ \begin{array}{c} I_{r \times (n-r)}(i_1|\cdots|i_r) \; | \;  L
\end{array} \right]
\left[ \begin{array}{c}
\ \\
I_{r \times (n-r)}(j_1|\cdots|j_r)^T \\
\ \\
\hline \\
N \\
\
\end{array} \right]
\end{align}
where $i_r \leq n-r$, $j_r \leq n-r$, $N \in M_r(Q)$ and $L \in GL_r(Q)$. Then there is a strong path from
\eqref{eqn_newK} to an entry $K$ in the region $(1 < 2 < \cdots < r) \times (1 < 2 < \cdots < r)$.
\end{lem}
\begin{proof}
\renewcommand{\arraystretch}{0.2}
\begin{figure}[t]
{\footnotesize
\[
\begin{array}{l||cccc}
            &       (j_1 < \cdots < j_r) & (j_1 < \cdots < j_{m-1} < m < j_{m+1} < \cdots < j_r)
  \\[0.5ex]
                &       \left[ \begin{array}{c} \ \\ I(j_1|\cdots|j_r)^T \\ \ \\ \hline \ \\ N \end{array} \right]
    &         \left[ \begin{array}{c} \ \\ I(j_1|\cdots|j_{m-1}|m,j_m|j_{m+1}|\cdots|j_r)^T \\ \ \\ \hline \ \\ N_1 \end{array} \right]
\\[4ex] \hline\hline
     & & & \\
 \left[ \begin{array}{c} I(i_1|\cdots|i_r) \; |L \end{array} \right]
&       K       &   K \\
& & & & \\
 I_{r \times n}(j_1|\cdots|j_r)
        &   I &         I \\
\end{array}
\]
}
\caption{Stage 1 of the proof of Lemma~\ref{lem_gettinginto1212}.}
\label{fig_stage1new}
\end{figure}
\renewcommand{\arraystretch}{1}
\renewcommand{\arraystretch}{0.2}
\begin{figure}[t]
{\footnotesize
\[
(j_1 < \cdots < j_{m-1} < m < j_{m+1} < \cdots < j_r)
\]
\[
\begin{array}{l||cccc}
&         \left[ \begin{array}{c} \ \\ I(j_1|\cdots|j_{m-1}|m,j_m|j_{m+1}|\cdots|j_r)^T \\ \ \\ \hline \ \\ N_1 \end{array} \right]
    &           \left[ \begin{array}{c} \ \\ I(j_1|\cdots|j_{m-1}|m|j_{m+1}|\cdots|j_r)^T \\ \ \\ \hline \ \\ N_2 \end{array} \right]
    \\[4ex] \hline\hline
     & & & \\
     \mathcal{I}
        &   K &         K \\
\end{array}
\]
}
\caption{Stage 2 of the proof of Lemma~\ref{lem_gettinginto1212}, where $\mathcal{I} = [ I(i_1|\cdots|i_r) |L ]$. This is a strong edge because it is contained in a single region.}
\label{fig_stage2new}
\end{figure}
\renewcommand{\arraystretch}{1}
The proof has two stages, first we find a strong path into the region
\begin{align}
\label{eqn_cat}
(i_1 < \cdots < i_r) \times (1 < \cdots < r)
\end{align}
and then apply
Lemma~\ref{lem_lemmaAnew} and
a dual argument to complete the proof.
To simplify notation in the proof let
\[
I(i_1|\cdots|i_r)
=
I_{r \times (n-r)}(i_1|\cdots|i_r).
\]
For the first stage we prove by induction on the order $\preceq$, defined in
Section~\ref{sec_preliminaries}, that there is a path into the region \eqref{eqn_cat}. If
\[
(j_1 < \cdots < j_r)
=
(1 < \cdots < r)
\]
then we are done, so suppose otherwise and let $j_m$ be the least $j_t$ such that $j_t \neq t$. Then there is a strong path with two edges from
\[
[I(i_1|\cdots|i_r)|L] \left[ \begin{array}{c} \ \\ I(j_1|\cdots|j_r)^T \\ \ \\ \hline N \end{array} \right]
\]
to
\begin{align}
\label{eqn_dino} [I(i_1|\cdots|i_r)|L] \left[ \begin{array}{c} \ \\ I(j_1|\cdots|j_{m-1}|m|j_{m+1}|\cdots|j_r)^T \\ \
\\ \hline N_2 \end{array} \right]
\end{align}
given in Figures~\ref{fig_stage1new} and \ref{fig_stage2new}. Here $N_1, N_2 \in M_r(Q)$ have been chosen in such a way
that the appropriate entries in the table are equal to $K$. Such choices for $N_1$ and $N_2$ are possible since $L$ is
invertible. This completes the proof of the first stage since
\begin{eqnarray*}
& & \{ j_1,  \ldots, j_{m-1}, m, j_{m+1}, \ldots, j_r \} \\
& \prec &
\{ j_1, \ldots, j_{m-1}, j_m, j_{m+1}, \ldots, j_r \}
\end{eqnarray*}
and so by induction there is a strong path from \eqref{eqn_dino} to an entry of the form
\begin{align}
\label{eqn_chicken} K = [I(i_1|\cdots|i_r)|L] \left[ \begin{array}{c} \ \\ I(1|2|\cdots|r)^T \\ \ \\ \hline Z
\end{array} \right].
\end{align}
Next by Lemma~\ref{lem_lemmaAnew} there is a strong path from \eqref{eqn_chicken} to
an entry
\begin{align}
\label{eqn_guitar} K = [I(i_1|\cdots|i_r)|L'] \left[ \begin{array}{c} \ \\ I(1|2|\cdots|r)^T \\ \ \\ \hline Z'
\end{array} \right]
\end{align}
where $Z' \in GL_r(Q)$ is invertible and $L'$ need not be. Then a dual argument to the one above gives a strong path
from \eqref{eqn_guitar} to an entry in the region $(1<\cdots < r) \times (1<\cdots < r)$, completing the proof of the
lemma. \end{proof}
Combining Lemmas~\ref{lem_frombefore}, \ref{lem_lemmaAnew} and \ref{lem_gettinginto1212} gives the following result
showing that we can strongly connect a large portion of the table $\mathbb{T}_{n,r}$.
This portion of the table is represented by the large light grey region in Figure~\ref{fig_table}.

\begin{cor}
\label{cor_thebigbox} For every matrix $K \in M_r(Q)$ the $\lambda$-graph of $K$ in $\mathbb{T}_{n,r}$ restricted to
\begin{align}
\label{eqn_bigbox}
\bigcup_{i_r \leq n-r, \; j_r \leq n-r}
(i_1 < i_2 < \cdots < i_r)
\times
(j_1 < j_2 < \cdots < j_r)
\end{align}
is strongly connected.
\end{cor}

We are now in a position to complete the proof of the main result of this section, Theorem~\ref{thm_thefulltable}. In
the following proof, for an $r \times n$  matrix $A$, with $r < n$, we shall use $A[i]$ to denote its $i$th column.
Dually, given an $n \times r$
matrix $B$, with $r < n$, we shall use $B[i]$ for its $i$th row.

One of the key steps in the following proof comes in the second paragraph where we define the number $t$ which we need to satisfy $t > r$ in order to establish linear dependence of a set of $t$ vectors in an $r$-dimensional vector space. For this argument to be valid we need
to make use of our assumption that $r < n/3$.

\begin{proof}[Proof of Theorem~\ref{thm_thefulltable}]
By Corollary~\ref{cor_thebigbox} it suffices to show that there is a strong path from every entry $K \in M_r(Q)$ in
$\mathbb{T}_{n,r}$ to an entry in the subtable \eqref{eqn_bigbox}. To this end, let $A \in \Y_r$, $B \in \X_r$ and let
$K = AB \in M_r(Q)$. Moreover, suppose that $LC(A) = \{i_1,i_2,\ldots, i_r \}$, $LR(B) = \{j_1,j_2,\ldots, j_r\}$ and
$i_r > n-r$. We shall prove that there exists a strong path from $(A,B)$ to $(A'',B')$ where the $LR(B')=LR(B)$ and
$LC(A'')$ is strictly less than $LC(A)$ in the \emph{lexicographic order} on the $r$-element subsets of
$\{1,\dots,n\}$. Using this it follows by induction that there is a strong path from $(A,B)$ to $(A_1,B_1)$ where
$LC(A_1)\subseteq\{1,\dots,n-r\}$ and $LR(B_1)=LR(B)$. Then by a dual argument there is a strong path from $(A_1,B_1)$
to $(A_2,B_2)$ where $LC(A_2)=LC(A_1)$, while $LR(B_2)\subseteq\{1,\dots,n-r\}$, and hence $(A_2,B_2)$ belongs to the
region \eqref{eqn_bigbox}.

Let
\[
\{k_1, \ldots, k_t \} = \{ 1,2,\ldots, n\} \setminus (LC(A) \cup LR(B))
\]
be the set of indices distinct from all indices of leading columns of $A$ and leading rows of $B$, where $k_1 <
k_2 < \cdots < k_t$. Since $r < n/3$ and $|LC(A)| = |LR(B)| = r$ it follows that
\[
t = | \{k_1, \ldots, k_t \} | = | \{ 1,2,\ldots, n\} \setminus (LC(A) \cup LR(B))| \geq n - 2r > r.
\]
Therefore since $t > r$, and the column space of $A$ has dimension $r$, it follows that there exists some $s$ such that column $k_s$ can be expressed as a (right) linear combination of the columns $\{k_{s+1}, k_{s+2}, \ldots, k_t \}$.
Write
\[
A[k_s] = A[k_{s+1}]\lambda_{s+1} + A[k_{s+2}]\lambda_{s+2} + \cdots + A[k_{t}]\lambda_{t},
\]
where $\lambda_i \in \F$.
Let $C$ be the $n \times r$ matrix defined by
\[
C[k_s] = -B[k_s], \quad
C[k_{s+j}] = \lambda_{s+j}B[k_{s}] \quad (1 \leq j \leq t-s),
\]
and all other rows of $C$ are set as the zero vector. Computing $AC$ we obtain
\begin{align*}
AC
&=
A[1]C[1] + A[2]C[2] + \cdots + A[n]C[n]
\\
&=
A[k_s]C[k_s] + A[k_{s+1}]C[k_{s+1}] + \cdots + A[k_n]C[k_n] + 0_{r \times r}
\\
&=
A[k_s](-B[k_s]) + A[k_{s+1}]\lambda_{s+1}B[k_{s}] + \cdots + A[k_t]\lambda_tB[k_s]
\\
&=
(-A[k_s] + A[k_{s+1}]\lambda_{s+1} + A[k_{s+2}]\lambda_{s+2}  + \cdots + A[k_t]\lambda_t)B[k_{s}] = 0_{r \times r}.
\end{align*}
Now define $B' = B+C$. We have
\[
B'[k_s] = B[k_s] + C[k_s] = 0_{1 \times r}.
\]
Moreover, $B'$ is an $n \times r$ matrix, whose transpose is in RRE form, and satisfies $LR(B') = LR(B)$. To see this,
consider an arbitrary $1 \times r$ row vector  of $ B' = B+C $ and compare it with the corresponding row of $B$. The
only rows that are different are those indexed by the $k_i$, where $i \geq s$, and all of the leading rows of $B$ are
left unchanged. Row $k_s$ of $B'$ is now the zero vector so this change certainly keeps the transpose of the matrix in RRE form. Now
consider some row $k_{s+j}$ with $1 \leq j \leq t-s$. By definition we have
\[
B'[k_{s+j}] = B[k_{s+j}] + \lambda_{s+j}B[k_s].
\]
Let $[b_1,b_2,\ldots,b_r]$ denote this row vector. Suppose that $b_v \neq 0$ for some $v$. Then it follows that either
column $v$ of $B[k_{s+j}]$ is non-zero or column $v$ of $B[k_s]$ must be non-zero which in turn, since the transpose of $B$ is in RRE
form, implies that the leading row $j_v$ of $B$ (that is, the first row of $B$ to have a non-zero term in the $v$th
column) must satisfy $j_v < k_{s+j}$. This argument shows that the transpose of $B'$ is in RRE form, and moreover that $LR(B')=LR(B)$.
Since
\[
AB' = A(B+C) = AB + AC = AB + 0 = AB = K,
\]
and $LR(B')=LR(B)$ it follows that there is a strong edge between $(A,B)$ and $(A,B')$ in $\mathbb{T}_{n,k}$ where $B'$
satisfies $B'[k_s] = 0_{1 \times r}$.

Next we claim that above we may also choose $k_s$ so that it satisfies $k_s < i_r$.
Indeed, suppose that $k_s > i_r$. Consider column $k_1$ of $A$. Certainly we have $k_1 \in \{1, \ldots, n-r\}$ since $r
< n/3$. Now let $A'$ be the matrix obtained by replacing column $A[k_s]$ of $A$ by a copy of $A[k_1]$ and leaving all
the other columns unchanged. Since $k_s > i_r$ the matrix $A'$ is still in RRE form, and since row $k_s$ of $B'$ is the
zero vector the product is not affected and we have $ AB = AB' = A'B'. $ Also $A'$ is in the same region as $A$ so the
corresponding $\lambda$-path is strong. Now column $k_1$ and column $k_s$ of $A'$ are equal and so column $k_1$ is a
linear combination of columns indexed by $\{ k_2, k_3, \ldots, k_t \}$. But $k_1 \leq n-r$ while $i_r \geq n-r+1$ by
assumption, and so $k_1 < i_r$. Therefore, running once again through the
argument given in the previous paragraph,
we may suppose without loss
of generality that $k_s < i_r$.

So now suppose that $k_s < i_r$ in $A$. Take the least $v$ such that $i_v > k_s$, which must exist since $k_s < i_r$.
Then define a matrix $A''$ obtained by replacing column $k_s$ of $A$ by a copy of $A[i_v]$ (i.e. the unit vector with
zeros everywhere except in position $v$). Clearly $A''$ is in RRE form. Again we see that $ AB = AB' = A''B', $ since
$B'[k_s]$ is the zero vector $0_{1 \times r}$. Also the edge between $AB'$ and $A''B'$ is easily seen to be a strong
edge by considering the column indexed by the scattered identity matrix $I(i_1, i_2, \ldots, i_r)^T$, and computing
\[
A \; I(i_1, i_2, \ldots, i_r)^T = A'' \; I(i_1, i_2, \ldots, i_r)^T = I_r.
\]
The proof is completed by observing that
\[
LC(A'') = \{ i_1,  i_2, \ldots, i_{v-1}, k_s, i_{v+1}, \ldots, i_r \}
\]
is strictly less than
\[
LC(A) = \{ i_1, i_2, \ldots, i_r \}
\]
in the lexicographic ordering on $r$-element subsets of $n$, and $LR(B') = LR(B)$.
\end{proof}

\section{Completing the proof}
\label{sec_uncovering}

By this stage in the proof we have succeeded in identifying all of the labels in the table with the corresponding
elements in the group $GL_r(\F)$. So after performing the identifications $f_{A,B} = f_{X,Y}$ whenever $BA = YX$ we
obtain a presentation $\mathcal{P}_{r,n}'$ with generators $\mathcal{F}' = \{ f_{U} : U \in GL_r(Q)  \}$.
As explained in the outline of the proof of the main theorem given in Section~\ref{sec_outline}, the mapping which sends each generator $f_U$ of $\mathcal{F}'$ to the matrix $U^{-1} \in GL_r(Q)$ defines a homomorphism from the group defined by $\mathcal{P}_{r,n}'$ onto the group $GL_r(Q)$. The following lemma shows that this map is actually an isomorphism by showing that every word over $\mathcal{F}'$ is equal to one of the generators.

\begin{lem}
\label{lem_multiplicationtable}
For every pair $A, B \in GL_r(Q)$
the relation
\[
f_B f_A = f_{AB}
\]
appears in the presentation $\mathcal{P}_{r,n}'$.
\end{lem}
\begin{proof}
We show this relation appears among the relations \eqref{eqn_bottom} by finding an appropriate singular square. Such a
singular square is illustrated in Figure~\ref{fig_theend}. This square is singular by \eqref{eq_99} since $A I_r^{-1} = (AB) B^{-1}$. This singular square then gives rise to the relation $f_A^{-1} = f_{AB}^{-1} f_B$ in \eqref{eqn_bottom}, or equivalently $f_B f_A = f_{AB}$, completing the proof.  \end{proof}

\renewcommand{\arraystretch}{0.2}
\begin{figure}[t]
\[
\begin{array}{cl||cc}
    &           &        \left[ \begin{array}{c}
    0_{r \times r} \\ \ \\ \hline \ \\
    0_{r \times r} \\ \ \\ \hline \ \\
    I_r \\ \ \\ \hline \ \\
    \\ 0_{(n-3r) \times r} \\ \ \end{array} \right]
    &  \left[ \begin{array}{c}
    I_r \\ \ \\ \hline \ \\
    0_{r \times r} \\ \ \\ \hline \ \\
    B \\ \ \\ \hline \ \\
    \\ 0_{(n-3r) \times r} \\ \ \end{array} \right]
    \\[8ex] \hline\hline
    & &  \\
  & \left[ \begin{array}{cccccccc} 0_{r \times r} & \vline & I_r & \vline & A & \vline & 0_{r \times (n-3r)} &
\end{array}
\right] &      A & AB \\
                          & & & \\
                          &
\left[ \begin{array}{cccccccc} 0_{r \times r} & \vline & 0_{r \times r} & \vline & I_r & \vline & 0_{r \times (n-3r)} &
\end{array} \right]
&       I_r
    & B \\
\end{array}
\]
\caption{The singular square for the proof of Lemma~\ref{lem_multiplicationtable}. This square is singular by Theorem
\ref{rectsing} and \eqref{eq_99}. Clearly all of these matrices are in RRE form.} \label{fig_theend}
\end{figure}
\renewcommand{\arraystretch}{1}

This completes all the steps of the proof of the main theorem, Theorem~\ref{thm_main}, as outlined in
Section~\ref{sec_outline}.

\section{Concluding Remarks}
\label{sec_concluding}

The obvious outstanding question that remains is whether our main result Theorem~\ref{thm_main} is true more generally
for rank $r$ components in the range $n/3 \leq r < n-1$. The proof given here certainly does not extend in a straightforward way
to higher values of $r$. In terms of the outline of the proof given in Section~\ref{sec_outline}, Stage~1 of the proof
does carry across and hold for all $r$ in the range $1 \leq r < n-1$. However, both Stages~2 and 3 of the proof the
assumption $r < n/3$ is used in the proofs. Let us now see that as they stand the results in these section do not
extend to the case $n/3 \leq r < n-1$. Recall that in Stage~2 we show the abstract generators can be identified in such
a way as to put them (after this identification) into bijective correspondence with the elements of $GL_r(Q)$.
Clearly a necessary condition for this to be possible is that the set $\mathcal{F}$ has size at least equal to the size of
$GL_r(Q)$. However in general, without the assumption $r < n/3$ it is not always true that the size of $\mathcal{F}$ is greater
than $GL_r(Q)$. This can be seen by a simple counting argument. Indeed, if $\mathbb{F}_q$ is the finite field with $q$
elements, then the number of $\gr$-classes in the $\gd$-class $D_r$ of $M_n(\mathbb{F}_q)$ is
precisely the number of $r$-dimensional subspaces of an $n$-dimensional vector space over $\mathbb{F}_q$ which is
given by the Gaussian
coefficient
\[
\left[
\begin{array}{c}
n \\ r
\end{array}
\right]_q
=\frac{
(q^n-1)(q^{n-1}-1) \cdots (q^{n-r+1}-1)
}
{
(q^r-1)(q^{r-1}-1)\cdots(q-1),
}
\]
and the number of idempotents in each $\gr$-class of $D_r$ is easily seen to be equal to $q^{r(n-r)}$. On the other hand the size of the
general linear group $GL_m(\mathbb{F}_q)$ is  well known to be given by the formula
\[
|GL_m(\mathbb{F}_q)|
=
(q^m-q^0)(q^m-q^1)\cdots (q^m-q^{m-1}).
\]
So for example if we take $n=7$ and $r=7-2=5$ and consider the $\gd$-class $D_5$ of $M_7(\mathbb{F}_2)$ then the number
of idempotents in $D_5$ is given by
\[
2^{10} \frac{
(2^7-1)(2^6-1)(2^5-1)(2^4-1)(2^3-1)
}{
(2^5-1)(2^4-1)(2^3-1)(2^2-1)(2-1)
}
=
\frac{2^{10}(2^7-1)(2^6-1)}{3},
\]
which is easily checked to be strictly less than the number of elements in the group $GL_5(\mathbb{F}_2)$ which,
using the above formula,
is equal to
\[
(2^5-1)(2^5-2)(2^5-2)(2^5-2^2)(2^5-2^3)(2^5-2^4).
\]
Therefore, if Theorem~\ref{thm_main} does extend to values $n/3 \leq r < n-2$ then the reason that the theorem holds is different for the reason that it holds for low rank $r < n/3$.

The main result of \cite{GR2} for the full transformation monoid $T_n$, and the main result Theorem~\ref{thm_main},
suggest that it may well be worth investigating maximal subgroups of endomorphism monoids of finite dimensional
independence algebras (in the sense of \cite{Cameron, Gould}) which form a class of monoids generalising both the full
transformation monoid and the full linear monoid over a field. In particular this may provide a route to proving a
common generalisation of Theorem~\ref{thm_main} and the corresponding result for $T_n$ established in \cite{GR2}.

\begin{acknowledgement}
The author Gray would like to gratefully acknowledge the support and kind hospitality of the University of Novi Sad
during research visits in the spring and autumn of 2011 where part of this research was undertaken.
The authors would like to thank an anonymous referee,
whose comments led to several significant improvements to the article.
We would also like to thank John Meakin for insightful correspondence and helpful suggestions.
\end{acknowledgement}

\bibliographystyle{amsplain}

\end{document}